\documentclass[12pt,amscd, amssymb]{amsart}
\usepackage{tikz}
\usepackage{pgfmath} 
\usepackage{pstricks-add}
\usetikzlibrary{arrows.meta, decorations.markings}
\usepackage{pgf,tikz}
\usetikzlibrary{arrows}
\usepackage{amssymb}
\usepackage{enumerate}
\usepackage{hyperref}
\theoremstyle{plain}

\newtheorem{Theorem}{Theorem}[section]
\newtheorem{Lemma}[Theorem]{Lemma}
\newtheorem{Corollary}[Theorem]{Corollary}
\newtheorem{Proposition}[Theorem]{Proposition}

\newtheorem*{Question}{Question}

\theoremstyle{definition}
\newtheorem{Definition}[Theorem]{Definition}

\newtheorem{Example}[Theorem]{Example}
\newtheorem{Remark}[Theorem]{Remark}

\def\reg{\operatorname{reg}}

\def\sk{\smallskip\par}

\def\deg{{\mathrm{deg}}}

\newcommand{\Z}{\mathbb{Z}}
\newcommand{\N}{\mathbb{N}}

\def\reg{\operatorname{reg}}

\def\dim{\operatorname{dim}}

\def\max{\operatorname{max}}

\begin{document}

\title{Independence Polynomials of graphs and degree of $h$-polynomials of edge ideals}

\email{}

\author{Ton That Quoc Tan}
\address{Department of Mathematics, FPT University, Danang
, Viet Nam}
\email{quoctanmath@gmail.com, tanttq@fe.edu.vn}
\date{}
 \date{}

\begin{abstract}
Let $G=(V,E)$ be a finite simple graph. In this paper, we study the degree of the $h$-polynomial of the edge ideal of $G$ in relation to the independence number of $G$. Our approach is based on the value of the independence polynomial of $G$ at $-1$ and its derivatives at $-1$. We establish a necessary and sufficient condition for the equality $\deg\ h_{R/I(G)}(t)=\alpha(G)$. As consequences, we obtain combinatorial formulas for the degree of the $h$-polynomial for several classes of graphs, including paths, cycles, bipartite graphs, Cameron--Walker graphs, and antiregular graphs.
\end{abstract}
\maketitle
\section*{Introduction} \sk
Let $G=(V,E)$ be a finite simple graph with vertex set $V=\{x_1,\ldots,x_n\}$. 
We identify the vertices of the graph with the corresponding variables in the polynomial ring $R = k[x_1,\ldots,x_n]$, where $k$ is a fixed field. 
The edge ideal of $G$, denoted by $I(G)$, is defined as the monomial ideal
\[
I(G) = (x_ix_j \mid \{x_i,x_j\} \in E) \subseteq R.
\]
It is well known that edge ideals of simple graphs reflect combinatorial aspects of algebraic invariants of monomial ideals such as the Castelnuovo--Mumford regularity, projective dimension, and others. 
Many studies have focused on expressing or bounding such invariants in terms of graph invariants; see, for example, \cite{TaiTuyl, Trung, Long}. 
In this paper, we investigate the degree of the $h$-polynomial of edge ideals, a fundamental yet less explored algebraic invariant.

The \emph{Hilbert function} of the graded ring $R/I(G)$, denoted by $HF_{R/I(G)}$, is the numerical function 
\[
HF_{R/I(G)} : \mathbb{Z}_{\ge 0} \longrightarrow \mathbb{Z}_{\ge 0}
\]
defined by
\[
HF_{R/I(G)}(i) = \dim_k (R/I(G))_i,
\]
where $(R/I(G))_i$ denotes the $i$-th graded component of $R/I(G)$, viewed as a $k$-vector space. 
The \emph{Hilbert series} of $R/I(G)$, denoted by $H_{R/I(G)}$, is defined by
\[
H_{R/I(G)}(t) = \sum_{i=0}^{\infty} HF_{R/I(G)}(i)\, t^i.
\]
According to \cite[Theorem 13.2]{Matsumura}, we have
\[
H_{R/I(G)}(t) = \frac{h_{R/I(G)}(t)}{(1-t)^d},
\]
where $h_{R/I(G)}(t) \in \mathbb{Z}[t]$, $h_{R/I(G)}(1) \neq 0$, and $d$ is the Krull dimension of $R/I(G)$.

The polynomial $h_{R/I(G)}(t)$ is called the \emph{$h$-polynomial} of $R/I(G)$. 
We denote the degree of this polynomial by $\deg\ h_{R/I(G)}(t)$. 
Recent results have focused on the relationship between the degree of the $h$-polynomial and other invariants such as the dimension, projective dimension, depth, and regularity of the edge ideal. 

In \cite{Hibi1}, the authors determined the possible values of the pair of invariants
 $$(\reg R/I(G), \deg\ h_{R/I(G)}(t))$$ for Cameron-Walker graphs. 
In \cite{Hibi2}, for any two integers $d,r \geq 1$, the authors constructed a graph $G$ such that 
\[
(\reg R/I(G), \deg\ h_{R/I(G)}(t)) = (d,r).
\]
Using the Hilbert series, Hibi et al.~\cite{Hibi3} proved that 
\[
\deg\ h_{R/I(G)}(t) = \dim R/I(G)
\]
for every Cameron--Walker graph. 
Recently, Biermann et al.~\cite{OKeefe} computed $\deg\ h_{R/I(G)}(t)$ in terms of the independence number for several fundamental classes of graphs, including paths, cycles, and bipartite graphs.

Motivated by these results, it is natural to consider the following question posed in \cite[Question 2]{OKeefe}.

\begin{Question}
For which classes of graphs does $\deg\ h_{R/I(G)}(t) = \alpha(G)$ hold?
\end{Question}

By exploiting the relationship between the $h$-polynomial of $R/I(G)$ and the independence polynomial of $G$, we give an affirmative answer to this question, as stated in the following theorem.  

\noindent {\bf Theorem \ref{general}}
{\it
Let $G=(V,E)$ be a finite simple connected graph. 
Assume that, using the elimination process, $G$ decomposes into a disjoint union of star graphs
$S_1,\ldots,S_r$ and graphs $G_1,\ldots,G_s$, where every vertex of each $G_j$ has degree at least $2$.
Then
$$
\deg\ h_{R/I(G)}(t) = \alpha(G)\ \ \text{if and only if} \ \ \ I(G_j,-1) \neq 0$$ for all $j=1,\ldots,s$.
}

As a consequence of Theorem \ref{general}, we obtain combinatorial formulas for the degree of $h$-polynomial of several fundamental graphs, including paths (Theorem \ref{pathmain}), cycles (Theorem \ref{cyclemain}), bipartite graphs (Theorem \ref{bipartite}), Cameron--Walker graphs (Theorem \ref{Cameron}), and antiregular graphs (Theorem \ref{antiregular}).

Section 1 introduces the notation used throughout the paper. 
In Section 2, we present a necessary and sufficient condition for 
$\deg h_{R/I(G)}(t) = \alpha(G)$ in terms of the value of the independence polynomial at $-1$. 
In Section 3, we characterize graphs for which the degree of the $h$-polynomial attains its maximum value. 
In Section 4, we apply Theorem~\ref{general} to several classes of graphs, including paths, cycles, bipartite graphs, Cameron--Walker graphs, and antiregular graphs, and obtain combinatorial formulas for the degree of the $h$-polynomial of these graphs.

\section{Preliminary}
\subsection{Graph Theory}
Let $G=(V,E)$ be a finite simple graph on the vertex set $V=\{x_1,\ldots,x_n\}$ and the edge set $E$. If $\{x_i,x_i\} \in E$, we say $x_i$ and $x_j$ are \emph{neighbors} in $G$. 
The \emph{neighborhood} of a vertex $x_i$ is the set $$N_G(x_i)=\{x_j \in V \mid \{x_i,x_j\} \in E\},$$
and $N_G[x_i] = N_G(x_i) \cup \{x_i\}$. If there is no ambiguity on $G$, we use $N(x_i)$ and $N[x_i]$, respectively. A vertex $x_i$ is a \emph{leaf vertex}, or simply \emph{leaf} if its neighborhood consists of exactly one vertex. A neighbor of a leaf is called \emph{whiskered}. The \emph{distance} $d(x_i,x_j)$ between two vertices $x_i$ and $x_j$ is the minimum 
length of all paths from $x_i$ to $x_j$. For a subset $U$ of vertices, we denote by 
$d(x_i,U)$ the distance from $x_i$ to $U$, that is,
\[
d(x_i,U)=\min\{d(x_i,x_j): x_j\in U\}.
\]  
If $X \subseteq V$, then $G[X]$ is the subgraph of $G$ induced by $X$. By $G-W$ we mean the subgraph $G[V-W]$, if $W \subseteq V$.
For brevity, we write $G-x_i$ instead of $G-\{x_i\}$, whenever $W= \{x_i\}$. We also denote by $G-F$ the subgraph of $G$ obtained by deleting the edges of $F$, for $F \subseteq E$, and we write shortly $G-x_ix_j$, whenever $F=\{\{x_i,x_j\}\}$. 

A subset $S \subseteq V$ is an \emph{independent set} of $G$ if $\{x_i,x_j\} \notin E$ for all $x_i, x_j \in S$. In particular, the empty set $\emptyset$ is an independent set of $G$. The \emph{independence number} of $G$, denoted by $\alpha(G)$, is the maximum size of an independent set of $G$. Let $s_i(G)$ be the number of the independent sets of size $i$ of $G$. The \emph{independence polynomial} of $G$ is defined by
$$I(G,t) = \sum_{i=0}^{\alpha(G)}s_i(G)t^i.$$
Let $G_1, G_2$ be two simple graphs. The \emph{disjoint union} of them, denoted by $G_1 \cup G_2$, is the graph with the vertex set $V(G_1 \cup G_2) = V(G_1) \cup V(G_2)$ and the edge set $E(G_1 \cup G_2) = E(G_1) \cup E(G_2)$. For two disjoint graphs, their \emph{Zykov sum}, denoted by $G_1 + G_2$, is the graph with vertex set $V(G_1 + G_2) = V(G_1) \cup V(G_2)$ and the edge set $E(G_1+G_2) = E(G_1) \cup E(G_2) \cup \{uv \mid u \in V(G_1), v \in V(G_2)\}$.

The following results provide useful recursive techniques for evaluating the independence polynomial of graphs.

\begin{Lemma}[\cite{Hoede}, Theorem 2.3]\label{Recusive}
Let $G=(V,E)$ be a simple graph, $x_i \in V$ and $x_ix_j \in E$. Then the following equalities hold:\\ 
{\rm (i)}  $I(G,t) = I(G - x_i,t) + tI(G - N[x_i],t)$;\\
{\rm (ii)} $I(G,t) = I(G - x_ix_j,t) - t^2I(G - (N(x_i) \cup N(x_j)),t)$.
\end{Lemma}
\begin{Lemma}[\cite{Hoede}, Corollary 3.8 and Theorem 3.9]\label{Joint}
Let $G_1, G_2$ be two vertex-disjoint graphs. Then\\
{\rm (i)} $I(G_1\cup G_2,t) = I(G_1,t)I(G_2,t)$;\\
{\rm (ii)} $I(G_1 +  G_2,t) = I(G_1,t) + I(G_2,t) -1.$
\end{Lemma}
A set of pairwise disjoint edges of $G$ is said to be a \emph{matching} of $G$. The \emph{matching number} of $G$ is the maximum cardinality of an matching of $G$, and it is denoted by $\mu(G)$. 
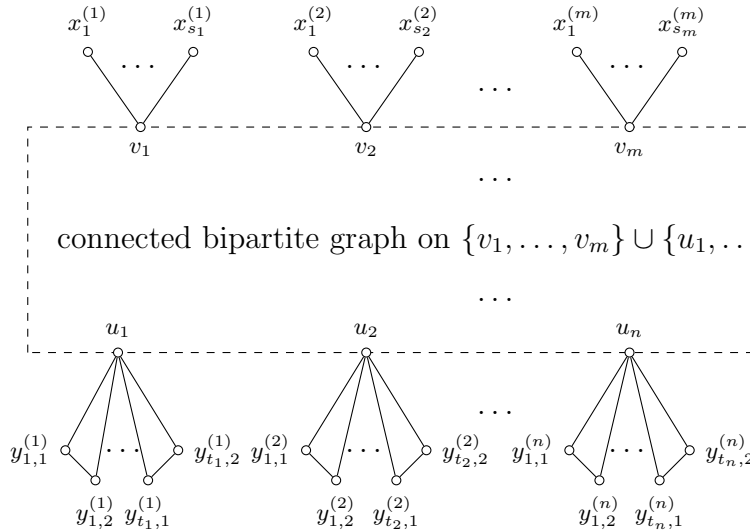
Let $\mathcal{M}$ be a matching of $G$. If the induced subgraph of $G$ on the vertices of $\mathcal{M}$ has edge set exactly $\mathcal{M}$, then $\mathcal{M}$ is called an \emph{induced matching}. The induced matching number of $G$  is the maximum cardinality of a induced matching of $G$, and it is denoted by $\nu(G)$. It is immediate that $\nu(G) \leq \mu(G)$. If $\nu(G) = \mu(G)$, then $G$ is called a \emph{Cameron-Walker graph}.
\subsection{Homological Invariants}
Let $R=k[x_1,\ldots,x_n]$ be the polynomial ring in $n$ variables over a field $k$ with $\deg\ x_i = 1$ for each $i$. For any ideal $I$ of $R$, the \emph{dimension} of $R/I$, denoted $\dim R/I$, is the length of the longest chain of prime ideals in $R/I$. 

If $I \subseteq R$ is a homogeneous ideal, then the \emph{Hilbert series} of $R/I$ is
$$H_{R/I}(t) = \sum_{i=0}^{\infty} \dim_k(R/I)_it^i,$$
where $(R/I)_i$ denotes the $i$-th graded piece of $R/I$. If $\dim R/I = d$, then the Hilbert series of $R/I$ is of the form
$$H_{R/I}(t) = \dfrac{h_0 + h_1t +h_2t^2 + \cdots + h_st^s}{(1-t)^d},$$
where each $h_i \in \Z$ (\cite[Theorem 13.2]{Matsumura}) and $h_{R/I}(1) \neq 0$. We say that
$$h_{R/I}(t) = h_0 + h_1t +h_2t^2 + \cdots + h_st^s$$
with $h_s \neq 0$ is the \emph{$h$-polynomial} of $R/I$. We call the difference $\deg\ h_{R/I}(t) - \dim R/I$ the \emph{$a$-invariant} of $R/I$ (\cite[Definition 4.4.4]{Herzog}) and denote it by $a(R/I)$.
\subsection{Edge ideals}
Let $G=(V,E)$ be a finite simple graph on $V=\{x_1,\ldots,x_n\}$. We associate with $G$ the square-free monomial ideal 
$$I(G) = \left(x_ix_j \mid \{x_i,x_j\} \in E\right) \subseteq R = k[x_1,\ldots,x_n].$$
The ideal $I(G)$ is called the \emph{edge ideal} of the graph $G$.

It is known that the dimension of $R/I(G)$ is given by
\begin{Lemma}
  $$\dim R/I(G) = \max\{|S| \mid S \text{ is an independent set of } G\} =\alpha(G).$$
\end{Lemma}
According to \cite[Proposition 6.2.1]{Herzog2}, the Hilbert series of $R/I(G)$ can be expressed as
$$H_{R/I(G)}(t) = \sum_{i=0}^{\alpha(G)}\dfrac{f_{i-1}t^i}{(1-t)^i},$$
where $f_{i-1}$ is the number of independent sets of cardinality $i$ in $G$. It follows that
$$\deg\ h_{R/I(G)}(t) \leq \alpha(G).$$
\begin{Remark}
  $a(R/I(G)) \leq 0.$ 
\end{Remark}
\section{Independence Polynomial at $-1$ and the Degree of the $h$-Polynomial}
In this section, we express the Hilbert series of $R/I(G)$ in terms of the independence polynomial of $G$ and its derivative evaluated at $-1$. 
As a consequence, we characterize when the degree of the $h$-polynomial of an edge ideal attains its maximum value $\alpha(G)$ in terms of the value of the independence polynomial of $G$ at $-1$.
In the following lemma, we give an explicit formula for the Hilbert series of $R/I(G)$ in terms of the independence polynomial of $G$.
\begin{Lemma}\label{Lemmamain}
  Let $G$ be a simple graph, $I(G)$ be the edge ideal of $G$ and $\alpha = \alpha(G)$ be the independence number of $G$. Let $I(G,t)$ be the independence polynomial of $G$. Then the Hilbert series of $R/I(G)$ is given by
  $$H_{ R/I(G)}(t)=\dfrac{1}{(1-t)^\alpha}\sum_{i=0}^{\alpha}\dfrac{I^{(i)}(G,-1)}{i!}(1-t)^{\alpha -i},$$
  where $I^{(i)}(G,-1)$ denotes the value of the $i$-th derivative of $I(G,t)$ evaluated at $-1$.
\end{Lemma} 
\begin{proof}
  By the claim in the proof of \cite[Theorem 3.2]{OKeefe}, we have
  $$H_{ R/I(G)}(t)=\dfrac{(1- \sum_{s=0}^{\alpha -1}D_s)(1-t)^\alpha + \sum_{s=0}^{\alpha - 1}D_s(1-t)^{\alpha-s-1}}{(1-t)^\alpha},$$
  where 
  $$D_s=\sum_{i=s+1}^{\alpha }(-1)^{i-1-s}s_i(G)\binom{i}{i-1-s}, \ \  0 \leq s \leq \alpha -1.$$
   In the proof of \cite[Theorem 3.2]{OKeefe}, we have 
   $$\sum_{s=0}^{\alpha -1}D_s = \sum_{i=1}^{\alpha}(-1)^{i-1}s_i(G).$$
   On the other hand,
   $$I(G,-1)=\sum_{i=0}^{\alpha}(-1)^is_i(G).$$
   Hence,
   $$\sum_{i=1}^{\alpha}(-1)^{i-1}s_i(G)= 1 - I(G,-1), $$
   since $s_0(G)=1$. Therefore,

   $$1- \sum_{s=0}^{\alpha -1}D_s = 1 - \sum_{i=1}^{\alpha}(-1)^{i-1}s_i(G) = I(G,-1).$$
   We note that for $0 \leq s \leq \alpha$, 
   \begin{align*}\label{1}
     I^{(s)}(G,x) & = \sum_{i=s}^{\alpha} (i-s+1)\cdots i\,s_i(G)x^{i-s} \\
                  & = \sum_{i=s}^{\alpha} s_i(G) s! \binom{i}{i-s} x^{i-s}. \\
   \end{align*}
   It follows that
   \begin{align*}
     \sum_{i=s}^{\alpha} s_i(G) \binom{i}{i-s} x^{i-s} & = \dfrac{I^{(s)}(G,x)}{s!} \\
     \end{align*}
     Thus,
     \begin{align*}
       D_s & =\sum_{i=s+1}^{\alpha}(-1)^{i-1-s}s_i(G)\binom{i}{i-1-s}  \\
        & =  \dfrac{I^{(s+1)}(G,-1)}{(s+1)!}.
     \end{align*}
    Consequently,
     $$\sum_{s=0}^{\alpha - 1}D_s(1-t)^{\alpha-s-1} = \sum_{s=1}^{\alpha }\dfrac{I^{(s)}(G,-1)}{s!}(1-t)^{\alpha -s}.$$
     This completes the proof.
\end{proof}
Using Lemma \ref{Lemmamain}, we characterize $\deg\ h_{R/I(G)}(t)$ in the following theorem.
\begin{Theorem}\label{chacracteristic}
  Let $G$ be a simple graph. Then $$\deg\, h_{R/I(G)}(t) = \alpha(G)\ \ \text{if and only if} \ \ \ I(G,-1) \neq 0.$$ Furthermore, if $I(G,-1) = 0$ and $k = \min \{ i \mid I^{(i)}(G,-1) \neq 0 \}$, then 
  $$\deg\, h_{R/I(G)}(t) = \alpha(G) - k.$$
\end{Theorem} 
\begin{proof}
  As a consequence of Lemma \ref{Lemmamain}, the $h$-polynomial of $R/I(G)$ has the form

  $$h_{R/I(G)}(t)= \sum_{i=0}^{\alpha}\dfrac{I^{(i)}(G,-1)}{i!}(1-t)^{\alpha -i}.$$
  Therefore, $\deg\, h_{R/I(G)}(t) = \alpha(G)$ if and only if $I(G,-1) \neq 0$. 
  Set $$A= \{i \mid I^{(i)}(G,-1) \neq 0\} \subset \N.$$ 
  Since $I^{(\alpha)}(G,-1) =  s_\alpha(G) \alpha ! \neq 0$, we have $A\neq \emptyset$. By the Well-Ordering Principle, there exists  $$k = \min \{ i \mid i \in A \}.$$
  Thus, if $I(G,-1) = 0$,  then 
   $\deg\, h_{R/I(G)}(t) = \alpha(G) - k$.
\end{proof}

\section{degree of $h$-polynomials of finite simple graphs}
Let $G=(V,E)$ be a finite simple graph. In this section, we answer Question~2 posed in \cite{OKeefe} by determining the classes of graphs for which $\deg h_{R/I(G)}(t)=\alpha(G)$.

Suppose that $G$ has several connected components. Then the degree of the $h$-polynomial of $G$ can be determined by studying the degrees of the $h$-polynomials of its connected components. This yields the following lemma.
\begin{Lemma}\label{cennected}
Let $G=\cup_{i=1}^kG_i$ be the disjoint union of the graphs $G_1,\ldots,G_k$. Then
$\deg\ h_{R/I(G)}(t) = \alpha(G)$ if and only if $\deg\ h_{R/I(G_i)}(t) = \alpha(G_i)$ for all $i=1,\ldots,k$. 
\end{Lemma}
\begin{proof}
  By Lemma \ref{Joint} (i) and induction, we have $$I(G,t)=\prod_{i=1}^{k}I(G_i,t).$$
Therefore,
  $I(G,-1) \neq 0$ if and only if $I(G_i,-1) \neq 0$ for all $i=1,\ldots,k$. According to Theorem \ref{chacracteristic}, the result follows.
\end{proof}

To study the degree of the $h$-polynomial of finite simple graphs, we first determine the degree of the $h$-polynomial of star graphs.
\begin{Lemma}\label{star}
  Let $S_n$ be a star graph where $n \geq 1$. Then $I(S_n,-1)=-1$ and
$\deg \ h_{R/I(S_n)}(t) = \alpha(S_n).$
\end{Lemma}
\begin{proof}
  Suppose that $x_0$ is the center vertex of $S_n$ and $x_1,\ldots,x_n$ are the $n$ leaves of $S_n$. By Lemma \ref{Recusive} (i), we have
  $$ I(S_n,t)  = I(S_n-x_0,t) + t I(S_n-N[x_0],t).$$
  Moreover, $I(S_n-N[x_0],t)=I(\emptyset,t)=1$, therefore $I(S_n-N[x_0],-1)=1$. By Lemma \ref{Joint} (i), we have
  $$I(S_n-x_0,t)=I(\{x_1\},t)\cdots I(\{x_n\},t)=(1+t)^n.$$
  Thus, $$ I(S_n,t)= (1+t)^n + t .$$ 
  Therefore, 
   $$ I(S_n,-1) = -1. $$
According to Theorem \ref{chacracteristic}, we have $\deg \ h_{R/I(S_n)} = \alpha(S_n).$
\end{proof}
In the following lemma, we describe two classes of graphs for which $\deg\ h_{R/I(G)} < \alpha(G)$. 
The first class consists of finite simple graphs with two leaf vertices at distance $3$, 
and the second class consists of graphs containing at least one isolated vertex.
 \begin{Lemma}\label{distance3}
   Let $G$ be a finite simple graph with two leaves at distance $3$ or with at least one isolated vertex. Then $I(G,-1)=0$ and $\deg\ h_{R/I(G)}(t) < \alpha(G)$.
 \end{Lemma}
\begin{proof}
If $G$ contains an isolated vertex, then $I(G,-1)=0$.
Suppose that $G$ contains two leaf vertices $u$ and $v$ such that $d(u,v)=3$, where $w \in N(u)$ and $y \in N(v)$. By Lemma~\ref{Recusive} (i), we have
\[
I(G,t) = I(G-w,t) + t\,I(G-N[w],t).
\]
Since $d(u,v)=3$, the graph $G-w$ contains an isolated vertex, namely $u$, and the graph
$G-N[w]$ also contains an isolated vertex, namely $y$. Therefore,
\[
I(G-w,-1)=0 \quad \text{and} \quad I(G-N[w],-1)=0.
\]
Hence, $I(G,-1)=0$. By Theorem \ref{chacracteristic}, it follows that
$\deg\ h_{R/I(G)}(t) < \alpha(G)$.
\end{proof}

By Lemma~\ref{cennected}, it suffices to study the value of $I(G,-1)$ for finite simple connected graphs. The following proposition gives the desired result.
\begin{Proposition}\label{generalfactor}
Let $G=(V,E)$ be a finite simple connected graph.
Define $G^{(0)}:=G$. For each $k\ge 0$, let $U_L^{(k)}$ be the set of leaves of $G^{(k)}$, and let
\[
U_2^{(k)}:=\{v\in V(G^{(k)}) : \deg_{G^{(k)}}(v)\ge 2 \text{ and } d_{G^{(k)}}(v,U_L^{(k)})=2\},
\]
where $G^{(k+1)}$ is the graph obtained from $G^{(k)}$ by deleting all vertices in
$U_2^{(k)}$ and all incident edges.

Then this process terminates after finitely many steps. Moreover, one of the following holds:

\begin{enumerate}
\item There exists a step $k$ such that $G^{(k)}$ contains an isolated vertex or two leaf vertices at distance $3$. In this case,
\[
I(G,-1)=0.
\]

\item Otherwise, the process terminates at a subgraph $G^{(m)}$, and the process decomposes $G$
into a disjoint union of stars $S_1,\ldots,S_r$ and graphs $G_1,\ldots,G_s$ in which all vertices
have degree at least $2$.
In this case,
\[
I(G,-1)=\prod_{i=1}^{r} I(S_i,-1)\prod_{j=1}^{s} I(G_j,-1).
\]
\end{enumerate}
\end{Proposition}
\begin{proof}
Let $U_L$ denote the set of leaves of $G$ and $U_2:=U_2^{(0)}(G)$ be the set of all non-leaf vertices of $G$ whose distance to some leaf is equal to $2$. If there exists a vertex in $U_2$ that is adjacent to a leaf vertex, then by Lemma
\ref{distance3} we have $I(G,-1)=0$. Hence, we may assume that no vertex in $U_2$ is adjacent to a leaf vertex.

Suppose that $U_2 = \{u_1,\ldots,u_m\}$. 
Set $H_i := H_{i-1} - u_i$ for $i=1,\ldots,m$, where $H_0 := G$.
By Lemma~\ref{Recusive} (i), we have
\begin{equation}\label{Plus4}
  I(H_{i-1},t) = I(H_i,t) + t\, I(H_{i-1}-N[u_i],t)
\end{equation}
for $i=1,\ldots,m$. Summing both sides of \eqref{Plus4} over $i=1,\ldots,m$, we obtain
\[
I(G,t) = I(H_m,t) + t \sum_{i=1}^{m} I(H_{i-1}-N[u_i],t).
\]
Moreover, for each $i=1,\ldots,m$, the graph $H_{i-1}-N[u_i]$ contains at least one isolated vertex. Hence,
\[
\sum_{i=1}^{m} I(H_{i-1}-N[u_i],-1) = 0.
\]
Furthermore, we observe that the graph $H_m = G - \{u_1,\ldots,u_m\}$ is the disjoint union of stars $S_1,\ldots,S_r$ and $G_1$. That is,
$$H_m = S_1 \cup \cdots \cup S_r \cup G_1.$$
By Lemma~\ref{Joint} (i), we have
\begin{equation}\label{plus5}
  I(G,-1) = \prod_{i=1}^{r} I(S_i,-1)\, I(G_1,-1).
\end{equation}
Consider the following two cases:\\
{\bf Case 1.} Suppose that $G_1$ contains at least one isolated vertex or two leaf vertices at distance $3$. By Lemma \ref{distance3}, 
we have $I(G,-1)=0$, and we terminate the process.\\
{\bf Case 2.} Assume that $G_1$ contains no isolated vertices and has no two leaf vertices at distance $3$.
 If  $G_1$  is disconnected with $k$ connected components $G_{11},\ldots,G_{1k}$, then by Lemma \ref{Joint} (i), we have
\[
I(G_1,-1) = \prod_{j=1}^{k} I(G_{1j},-1).
\]
Hence, $I(G_1,-1) \neq 0$ if and only if  $I(G_{1j},-1) \neq 0$ for all $j=1,\ldots,k$. Thus,  without loss of generality, we may assume that $G_1$ is connected.  
If $U_2(G_1)\neq \emptyset$, we continue to apply the process to $G_1$. 
The process described above terminates after finitely many steps. 
Consequently, after finitely many iterations, the resulting graph either contains an isolated vertex or has two leaf vertices at distance $3$ or has no vertex whose distance to a leaf is equal to $2$.
Without loss of generality, we may assume that the process terminates at the graph $G_1$. 
Then $G_1$ is either a star or a graph in which all vertices have degree at least $2$. Indeed, if $G_1$ is a star or a graph in which all vertices have degree at least $2$, then $U_2(G_1)=\emptyset$. 
Conversely, suppose that $U_2(G_1)=\emptyset$ and that $G_1$ is neither a star nor a graph in which all vertices have degree at least $2$. 
If $G_1$ contains a leaf, say $v$. 
Let $w$ be a vertex of $G_1$ such that $d(v,w)=1$. 
Since $G_1$ is not a star, there exists a vertex $x\neq v$ such that $d(w,x)=1$ and $x$ is not a leaf. 
Thus, $x$ is a non-leaf vertex satisfying $d(x,v)=2$, which implies that $x\in U_2(G_1)$, a contradiction to the assumption that $U_2(G_1)=\emptyset$. 
Therefore, $G_1$ must be either a star or a graph in which all vertices have degree at least $2$.
Hence,
\[
I(G,-1) = \prod_{i} I(S_i,-1)\, \prod_{j} I(G_j,-1),
\]
where each $G_j$ is a graph in which all vertices have degree at least $2$ and each $S_i$ is a star graph.
\end{proof}
We call the process described in Proposition~\ref{generalfactor} the \emph{elimination process}.\\
The following example illustrates this elimination process.
\begin{Example}
Let $G$ be the graph shown below, with $U_2(G)=\{u_1,\ldots,u_m\}$.
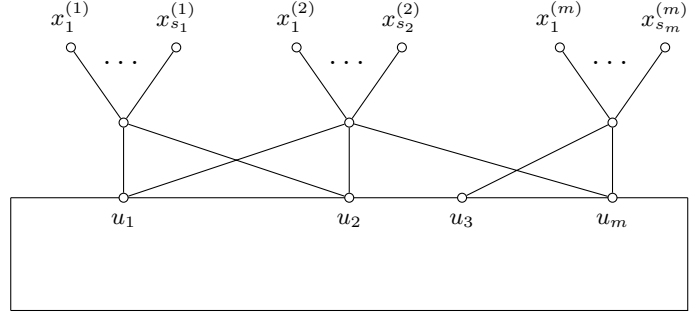
\begin{center}
\begin{tikzpicture}[
    dot/.style={circle, draw, fill=white, inner sep=1.2pt},
    every label/.style={font=\scriptsize}
]


    \coordinate (LL) at (-1, -1);
    \coordinate (LR) at (8, -1);
    \coordinate (UL) at (-1, 0.5);
    \coordinate (UR) at (8, 0.5);

    \draw (LL) -- (LR) -- (UR); 
    \draw (LL) -- (UL);

    \draw (UL) -- (UR);
    \node at (4.5, 0) {};

    \node[dot] (v1) at (0.5, 1.5) {};
    \node[dot] (v2) at (3.5, 1.5) {};
    \node[dot] (vr) at (7, 1.5) {};

    \node[dot, label=below:$u_1$] (u1) at (0.5, 0.5) {};
    \node[dot, label=below:$u_2$] (u2) at (3.5, 0.5) {};
    \node[dot, label=below:$u_3$] (u3) at (5, 0.5) {};
    \node[dot, label=below:$u_m$] (um) at (7, 0.5) {};

    \node[dot, label=above:$x_1^{(1)}$] (x11) at (-0.2, 2.5) {};
    \node[dot, label=above:$x_{s_1}^{(1)}$] (x1s) at (1.2, 2.5) {};
    \draw (v1) -- (x11); \draw (v1) -- (x1s);
    \node at (0.5, 2.3) {$\dots$};

    \node[dot, label=above:$x_1^{(2)}$] (x21) at (2.8, 2.5) {};
    \node[dot, label=above:$x_{s_2}^{(2)}$] (x2s) at (4.2, 2.5) {};
    \draw (v2) -- (x21); \draw (v2) -- (x2s);
    \node at (3.5, 2.3) {$\dots$};

    \node[dot, label=above:$x_1^{(m)}$] (xm1) at (6.3, 2.5) {};
    \node[dot, label=above:$x_{s_m}^{(m)}$] (xms) at (7.7, 2.5) {};
    \draw (vr) -- (xm1); \draw (vr) -- (xms);
    \node at (7, 2.3) {$\dots$};
    \draw (u1) -- (v1) (u1) -- (v2) (u2) -- (v1) (u2) -- (v2) (u3) -- (vr) (um) -- (vr) (um) -- (v2); 

\end{tikzpicture}\\
\end{center}
The subgraph $G-N[u_1]$ is shown below.
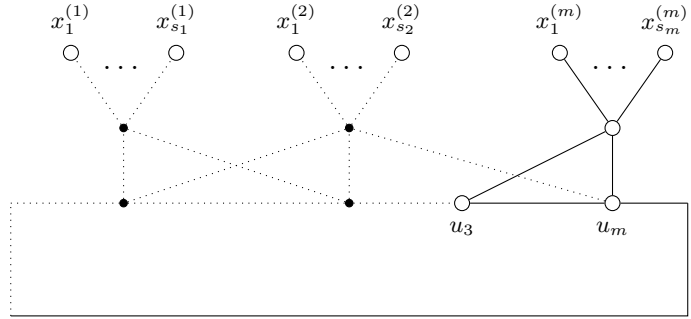
\begin{center}
\begin{tikzpicture}[
    dot/.style={circle, draw, fill=white, inner sep=2pt},
    every label/.style={font=\scriptsize},
ring/.style={circle, draw, fill=black, inner sep=1pt}
]


    \coordinate (LL) at (-1, -1);
    \coordinate (LR) at (8, -1);
    \coordinate (UL) at (-1, 0.5);
    \coordinate (UR) at (8, 0.5);

    \draw (LL) -- (LR) -- (UR); 
    \draw[dotted] (LL) -- (UL);

    \node at (4.5, 0) {};

    \node[ring] (v1) at (0.5, 1.5) {};
    \node[ring] (v2) at (3.5, 1.5) {};
    \node[dot] (vr) at (7, 1.5) {};

    \node[ring] (u1) at (0.5, 0.5) {};
    \node[ring] (u2) at (3.5, 0.5) {};
    \node[dot, label=below:$u_3$] (u3) at (5, 0.5) {};
    \node[dot, label=below:$u_m$] (um) at (7, 0.5) {};

    \node[dot, label=above:$x_1^{(1)}$] (x11) at (-0.2, 2.5) {};
    \node[dot, label=above:$x_{s_1}^{(1)}$] (x1s) at (1.2, 2.5) {};
    \draw[dotted] (v1) -- (x11); \draw[dotted] (v1) -- (x1s);
    \node at (0.5, 2.3) {$\dots$};

    \node[dot, label=above:$x_1^{(2)}$] (x21) at (2.8, 2.5) {};
    \node[dot, label=above:$x_{s_2}^{(2)}$] (x2s) at (4.2, 2.5) {};
    \draw[dotted] (v2) -- (x21); \draw[dotted] (v2) -- (x2s);
    \node at (3.5, 2.3) {$\dots$};

    \node[dot, label=above:$x_1^{(m)}$] (xm1) at (6.3, 2.5) {};
    \node[dot, label=above:$x_{s_m}^{(m)}$] (xms) at (7.7, 2.5) {};
    \draw (vr) -- (xm1); \draw (vr) -- (xms);
    \node at (7, 2.3) {$\dots$};
    \draw 
(u3) -- (vr) (um) -- (vr) ; 
 \draw[dotted](u1) -- (v1) (u1) -- (v2) (um) -- (v2) (u2) -- (v1) (u2) -- (v2) (u1)-- (u2) (u2) -- (u3) (UL) -- (u1);
\draw (u3)-- (um) (um) -- (UR);
\end{tikzpicture}\\
\end{center}
The subgraph $H_1 := G - u_1$ is shown below.
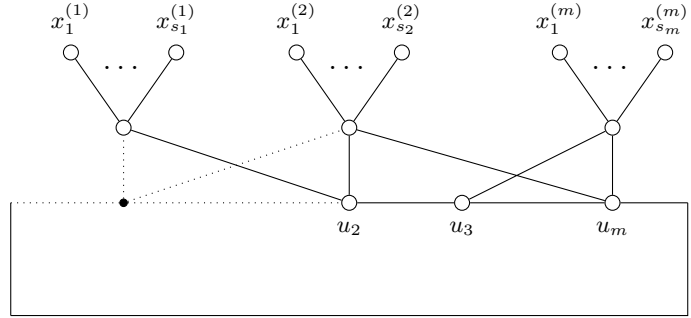
\begin{center}
\begin{tikzpicture}[
    dot/.style={circle, draw, fill=white, inner sep=2pt},
    every label/.style={font=\scriptsize},
ring/.style={circle, draw, fill=black, inner sep=1pt}
]


    \coordinate (LL) at (-1, -1);
    \coordinate (LR) at (8, -1);
    \coordinate (UL) at (-1, 0.5);
    \coordinate (UR) at (8, 0.5);

    \draw (LL) -- (LR) -- (UR); 
    \draw (LL) -- (UL);

    \node at (4.5, 0) {};

    \node[dot] (v1) at (0.5, 1.5) {};
    \node[dot] (v2) at (3.5, 1.5) {};
    \node[dot] (vr) at (7, 1.5) {};

    \node[ring] (u1) at (0.5, 0.5) {};
    \node[dot, label=below:$u_2$] (u2) at (3.5, 0.5) {};
    \node[dot, label=below:$u_3$] (u3) at (5, 0.5) {};
    \node[dot, label=below:$u_m$] (um) at (7, 0.5) {};

    \node[dot, label=above:$x_1^{(1)}$] (x11) at (-0.2, 2.5) {};
    \node[dot, label=above:$x_{s_1}^{(1)}$] (x1s) at (1.2, 2.5) {};
    \draw (v1) -- (x11); \draw (v1) -- (x1s);
    \node at (0.5, 2.3) {$\dots$};

    \node[dot, label=above:$x_1^{(2)}$] (x21) at (2.8, 2.5) {};
    \node[dot, label=above:$x_{s_2}^{(2)}$] (x2s) at (4.2, 2.5) {};
    \draw (v2) -- (x21); \draw (v2) -- (x2s);
    \node at (3.5, 2.3) {$\dots$};

    \node[dot, label=above:$x_1^{(m)}$] (xm1) at (6.3, 2.5) {};
    \node[dot, label=above:$x_{s_m}^{(m)}$] (xms) at (7.7, 2.5) {};
    \draw (vr) -- (xm1); \draw (vr) -- (xms);
    \node at (7, 2.3) {$\dots$};
    \draw 
(u3) -- (vr) (um) -- (vr) (u2) -- (v1) (u2) -- (v2) (um) -- (v2); 
 \draw[dotted](u1) -- (v1) (u1) -- (v2)   ;
\draw (u3)-- (um) (um) -- (UR) (u2) -- (u3);
\draw[dotted] (u1)-- (u2)  (UL) -- (u1);
\end{tikzpicture}\\
\end{center}
Similarly, the subgraph $H_{i-1}-N[u_i]$ for $i=1,\ldots,m$ also contains at least one isolated vertex.
Finally, define $H_m := G-\{u_1,\ldots,u_m\}$.  Then $H_m$ is a disjoint union of star graphs and $G_1$.
\begin{center}
\begin{tikzpicture}[
    dot/.style={circle, draw, fill=white, inner sep=2pt},
    every label/.style={font=\scriptsize},
ring/.style={circle, draw, fill=black, inner sep=1pt}
]


    \coordinate (LL) at (-1, -1);
    \coordinate (LR) at (8, -1);
    \coordinate (UL) at (-1, 0.5);
    \coordinate (UR) at (8, 0.5);

    \draw (LL) -- (LR) -- (UR); 
    \draw (LL) -- (UL);

    \node at (4.5, 0) {};

    \node[dot] (v1) at (0.5, 1.5) {};
    \node[dot] (v2) at (3.5, 1.5) {};
    \node[dot] (vr) at (7, 1.5) {};

    \node[ring] (u1) at (0.5, 0.5) {};
    \node[ring] (u2) at (3.5, 0.5) {};
    \node[ring] (u3) at (5, 0.5) {};
    \node[ring] (um) at (7, 0.5) {};

    \node[dot, label=above:$x_1^{(1)}$] (x11) at (-0.2, 2.5) {};
    \node[dot, label=above:$x_{s_1}^{(1)}$] (x1s) at (1.2, 2.5) {};
    \draw (v1) -- (x11); \draw (v1) -- (x1s);
    \node at (0.5, 2.3) {$\dots$};

    \node[dot, label=above:$x_1^{(2)}$] (x21) at (2.8, 2.5) {};
    \node[dot, label=above:$x_{s_2}^{(2)}$] (x2s) at (4.2, 2.5) {};
    \draw (v2) -- (x21); \draw (v2) -- (x2s);
    \node at (3.5, 2.3) {$\dots$};
 \node[draw=none, fill=none] at (3.5, -0.5) { $G_1$};
    \node[dot, label=above:$x_1^{(m)}$] (xm1) at (6.3, 2.5) {};
    \node[dot, label=above:$x_{s_m}^{(m)}$] (xms) at (7.7, 2.5) {};
    \draw (vr) -- (xm1); \draw (vr) -- (xms);
    \node at (7, 2.3) {$\dots$};

 \draw[dotted](u1) -- (v1) (u1) -- (v2)  (u2) -- (v1) (u2) -- (v2) (u2) -- (u3)  (u3)-- (um) (um) -- (UR) (u3) -- (vr) (um) -- (vr)   (um) -- (v2) (u1)-- (u2)  (UL) -- (u1);

\end{tikzpicture}\\
\end{center}
\end{Example}

We are now in a position to derive the main result of this section.
 \begin{Theorem}\label{general}
Let $G=(V,E)$ be a finite simple connected graph. 
Suppose that the elimination process decomposes $G$ into the disjoint union of star graphs
$S_1,\ldots,S_r$ and graphs $G_1,\ldots,G_s$, where every vertex of each $G_j$ has degree at least $2$.
Then
$$
\deg\ h_{R/I(G)}(t) = \alpha(G)\ \ \text{if and only if} \ \ \ I(G_j,-1) \neq 0$$ for all $j=1,\ldots,s$.
 \end{Theorem}
\begin{proof}
Using the elimination process, $G$ decomposes into a disjoint union of stars
$S_1,\ldots,S_r$ and graphs $G_1,\ldots,G_s$, where every vertex of each $G_j$ has degree at least $2$.
By Proposition \ref{generalfactor} (2), we have
\[
I(G,-1)=\prod_{i=1}^{r} I(S_i,-1)\prod_{j=1}^{s} I(G_j,-1).
\]
According to Lemma \ref{star}, we have $I(S_i,-1)=-1\neq 0$ for all $i$.
Therefore,
\[
I(G,-1)\neq 0
\quad \text{if and only if} \quad
\prod_{j=1}^{s} I(G_j,-1)\neq 0,
\]
which holds if and only if $I(G_j,-1)\neq 0$ for all $j=1,\ldots,s$.
This completes the proof.
\end{proof}

\section{Applications to Some Fundamental Graphs}
In this section, we apply the elimination process to derive combinatorial formulas for the degree of the $h$-polynomial of several familiar classes of graphs, including paths, cycles, multipartite graphs, Cameron-Walker graphs, and antiregular graphs.
\subsection{Degree of $h$-polynomials of paths and cycles}
In this subsection, we determine the degree of the $h$-polynomial of the edge ideal 
of paths and cycles. We first consider the case of paths.

\begin{Theorem}\label{pathmain}
  Let $P_n$ be the path on $n \geq 1$ vertices. Then 
   $$\deg\ h_{R/I(P_n)} = \alpha(P_n) \ \ \ \text{if and only if} \ \ \ n \equiv 0, 2 \pmod 3,$$
   and 
   $$\deg\ h_{R/I(P_n)} = \alpha(P_n) - 1 \ \ \ \text{if and only if} \ \ \ n \equiv 1 \pmod 3.$$
  
\end{Theorem}
\begin{proof}
 We consider the following three cases.\\
  {\bf Case 1.}  Let $n=3k$ with $k \ge 0$. Applying the elimination process recursively, we remove the vertices $x_3,\ldots,x_{3k-3}$. The resulting subgraph is a disjoint union of $k$ stars, consisting of $k-1$ copies of $P_2$ and one copy of $P_3$. By Theorem~\ref{general}, we have $\deg\ h_{R/I(P_n)}(t) = \alpha(P_n)$. Note that $I(P_2,t)=1+2t$ and $I(P_3,t)=1+3t+t^2$. By Proposition~\ref{generalfactor} (2), we obtain
\begin{equation}\label{3k}
I(P_n,-1)=I(P_2,-1)^{k-1}I(P_3,-1)=(-1)^k=(-1)^n.
\end{equation}
  {\bf Case 2.}  Let $n=3k +2$ with $k \ge 0$. Applying the elimination process recursively, we remove the vertices $x_3,\ldots,x_{3k}$. The resulting subgraph is a disjoint union of $k+1$ stars $P_2$. By Theorem~\ref{general}, we have $\deg\ h_{R/I(P_n)}(t) = \alpha(P_n)$.
  By Proposition~\ref{generalfactor} (2), we obtain
\begin{equation}\label{3k+2}
I(P_n,-1)=I(P_2,-1)^{k}=(-1)^{k+1}= (-1)^{n-1}.
\end{equation}
  {\bf Case 3.} Let $n=3k +1$ with $k \ge 0$. Applying the elimination process recursively, we remove the vertices $x_3,\ldots,x_{3k-3}$. The resulting subgraph is a disjoint union of $k$ stars, consisting of $k-1$ copies of $P_2$ and one copy of $P_4$. Since $P_4$ has two leaves of distance 3, by Proposition \ref{generalfactor} (1) we have $I(P_n,-1) = 0$. 
We claim that 
$$I'(P_n,-1)=\left(\frac{n+2}{3}\right)(-1)^{n-1}$$
for $n=3k+1$. We prove the claim by induction on $k$.
When $k=0$, we have $I(P_1,t) = 1 + t$, therefore $I'(P_1,t) = 1$. It follows that $I'(P_1,-1)=1=\left(\dfrac{1+2}{3}\right)(-1)^{1-1}$. Assume that the statement is true for all $i<k$, that is $I'(P_{3i+1},-1) = \left(i+1\right)(-1)^{3i}$, for all $i < k$. By Lemma \ref{Recusive} (ii), we have
$$I(P_{3k+1},t)=(t+1)I(P_{3k},t) - t^2I(P_{3(k-1)+1},t).$$
Differentiating both sides, we obtain
$$I'(P_{3k+1},t)=I(P_{3k},t) + (t+1)I'(P_{3k},t) - t^2I'(P_{3(k-1)+1},t) - 2tI(P_{3(k-1)+1},t). $$
We have $I(P_{{3(k-1)+1}},-1) = 0$. By substituting  $ t = -1 $ into the expression, we get 
$$I'(P_n,-1) = I'(P_{3k+1},-1) = I(P_{3k},-1) - (-1)^2I'(P_{3(k-1)+1},-1). $$
By the induction hypothesis and by (\ref{3k}) we have $I(P_{3k},-1) = (-1)^{3k}$. It follows that
$$I'(P_n,-1) = (-1)^{3k} + (-1)^3k(-1)^{3(k-1)} = (k+1)(-1)^{3k}=\left(\dfrac{n+2}{3}\right)(-1)^{n-1}.$$
By Theorem \ref{chacracteristic}, we have $\deg\ h_{R/I(P_n)}(t) = \alpha(P_n) -1$.
This completes the proof of the theorem.
\end{proof}
From Theorem \ref{pathmain}, we obtain the following corollary.
\begin{Corollary}
  Let $P_n$ be the path on $n \geq 1$ vertices. Then, $a(R/I(P_n)) \in \{-1, 0\}$.
\end{Corollary}
We next determine the degree of the $h$-polynomial for cycles.

In the following lemma, we compute the value of the independence polynomial of a cycle graph at $-1$.
\begin{Lemma}\label{cycle}
  Let $C_n$ be the cycle graph on $n \geq 3$ vertices and $I(C_n,t)$ be the independence polynomial of $C_n$. Then
$$I(C_n,-1) = \left\{
    \begin{array}{ll}
     2.(-1)^{n}  & \hbox{ if } n \equiv 0 \pmod 3; \\
      (-1 )^{n-1} & \hbox{ if } n \equiv 1,2 \pmod 3.
    \end{array}
  \right.
$$
\end{Lemma}
\begin{proof}
By Lemma \ref{Recusive}, we have
$$I(C_n,t) = I(P_n,t) - t^2I(P_{n-4},t).$$
{\bf Case 1.} If $n \equiv 0 \pmod 3$, then $n - 4 \equiv 2 \pmod 3$. According to (\ref{3k}), we have $I(P_n,-1) = (-1)^n$ and $I(P_{n-4},-1) = (-1)^{n-3}$. It follows that
$$I(C_n,-1) = (-1)^n - (-1)^2(-1)^{n-3} = 2(-1)^n.$$
{\bf Case 2.} If $n \equiv 1 \pmod 3$, then $n - 4 \equiv 0 \pmod 3$. By Case 2 of the proof of Theorem \ref{pathmain}, we have $I(P_n,-1) = 0$ and $I(P_{n-4},-1) = (-1)^{n-4}$.
It follows that
$$I(C_n,-1) =  - (-1)^2(-1)^{n-4} = (-1)^{n-1}.$$
{\bf Case 3.} If $n \equiv 2 \pmod 3$, then $n - 4 \equiv 1 \pmod 3$. By (\ref{3k+2}), we get $I(P_n,-1) = (-1)^{n-1}$ and $I(P_{n-4},-1) = 0$.
It follows that
$$I(C_n,-1) = (-1)^{n-1}.$$
This completes the proof.
\end{proof}
Combining Theorem~\ref{chacracteristic} and Lemma~\ref{cycle}, we get the following theorem.
\begin{Theorem}\label{cyclemain}
  Let $C_n$ be the cycle on $n \geq 3$ vertices. Then degree of the $h$-polynomial of $R/I(C_n)$ is
$$\deg\ h_{R/I(C_n)} =  \alpha(C_n).$$
\end{Theorem}
\begin{proof}
By Lemma \ref{cycle}, we have $I(C_n,-1) \neq 0$ for all $ n\geq 3$. According to Theorem \ref{chacracteristic}, we get $\deg\ h_{R/I(C_n)} =  \alpha(C_n).$
\end{proof}
From Theorem \ref{cyclemain}, we obtain the following corollary.
\begin{Corollary}
  Let $C_n$ be the path on $n \geq 3$ vertices. Then, $a(R/I(C_n)) = 0$.
\end{Corollary}
\subsection{Degree of $h$-polynomials complete multipartite graphs and bipartite graphs}
In this section, we study the degree of the $h$-polynomial of complete multipartite graphs and bipartite graphs. In particular, we establish a necessary and sufficient condition for the degree of the $h$-polynomial of a bipartite graph to equal its independence number by applying Theorem \ref{general}.

\begin{Lemma}\label{Zykov}
Let $G=\sum_{i=1}^kG_i$ be the Zykov sum of the disjoint graphs $G_1,\ldots,G_k$. Then
$\deg\ h_{R/I(G)}(t) = \alpha(G)$ if and only if $\sum_{i=1}^{k}I(G_i,-1) \neq k-1$. 
\end{Lemma}
\begin{proof}
  By Lemma \ref{Joint} (ii) and induction, we have $$I(G,t)=\sum_{i=1}^{k}I(G_i,t) - (k-1).$$ Therefore,
  $I(G,-1) \neq 0$ if and only if $\sum_{i=1}^{k}I(G_i,-1) \neq k-1$. According to Theorem \ref{chacracteristic}, the conclusion follows.
\end{proof}
We observe that a complete $q$-partite graph is the Zykov sum of $q$ complements of complete graphs. 
From Lemma~\ref{Zykov}, we obtain the following theorem.

\begin{Theorem}
Let $G=K_{m_1,m_2,\ldots,m_q}$ be the completed $q$-partite graph with $m_1 \geq m_2 \geq \cdots \geq m_q \geq 1$. Then
$$\deg\ h_{R/I(G)}(t) = m_1.$$
In particular, if $K_{a,b}$ with $a \geq b \geq 1$ is a complete bipartite graph, then 
$$\deg\ h_{R/I(K_{a,b})}=a.$$
\end{Theorem}
\begin{proof}
  We note that $K_{m_1,m_2,\ldots,m_q} = \overline{K}_{m_1} +\overline{ K}_{m_2} + \cdots +\overline{ K}_{m_q}$. But for every $i=1, \ldots, q$, we have $I(\overline{K}_{m_i},t)=(1+t)^{m_i}$. Therefore, $$I(\overline{K}_{m_1},-1) + \cdots + I(\overline{K}_{m_q},-1) = 0 \neq  q - 1.$$
  By Lemma \ref{Zykov},
 $$\deg\ h_{R/I(K_{m_1,m_2,\ldots,m_n})}(t) = \alpha(K_{m_1,m_2,\ldots,m_n}) = m_1.$$
\end{proof}
\begin{Corollary}
  Let $K_{m_1, \ldots, m_q}$ be a completed $q$-partite graph. Then 
  $a(R/I(K_{m_1, \ldots, m_q})) = 0.$
\end{Corollary}

Applying Theorem \ref{general}, if $G$ is a connected bipartite graph, then the following theorem holds.
\begin{Theorem}\label{bipartite}
Let $G$ be a connected bipartite graph.
Assume that, using the reduction process, $G$ decomposes into a disjoint union of star graphs
$S_1,\ldots,S_r$ and bipartite graphs $G_1,\ldots,G_s$, where every vertex of each $G_j$ has degree at least $2$.
Then
\[
\deg\ h_{R/I(G)}(t)=\alpha(G)
\]
if and only if $I(G_j,-1)\neq 0$ for all $j=1,\ldots,s$.
\end{Theorem}

\begin{Example}
  Let $G$ be the connected bipartite graph as follows.
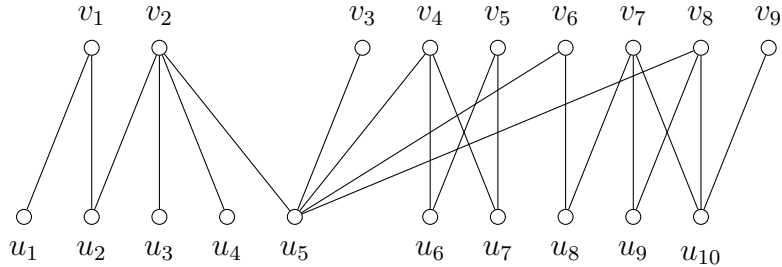
\begin{center}
  \begin{tikzpicture}[scale=0.9, every node/.style={draw=black, fill=white, circle, inner sep=2pt},
 dot/.style={circle, draw, fill=white, inner sep=2pt}
]
    \node[label=below:$u_1$] (u1) at (0, 0) {};
    \node[label=above:$v_1$] (v1) at (1, 2.5) {};
    \node[label=below:$u_2$] (u2) at (1, 0) {};
    \node[label=below:$u_4$] (u4) at (3, 0) {};
    \node[label=above:$v_2$] (v2) at (2, 2.5) {};
    \node[label=below:$u_3$] (u3) at (2, 0) {}; 
    \node[label=below:$u_5$] (u5) at (4, 0) {}; 
    \node[label=above:$v_3$] (v3) at (5, 2.5) {};
    \node[label=above:$v_4$] (v4) at (6, 2.5) {};
    \node[label=above:$v_5$] (v5) at (7, 2.5) {};
    \node[label=above:$v_6$] (v6) at (8, 2.5) {};
    \node[label=above:$v_7$] (v7) at (9, 2.5) {};
    \node[label=above:$v_8$] (v8) at (10, 2.5) {};
    \node[label=above:$v_9$] (v9) at (11, 2.5) {};
    \node[label=below:$u_6$] (u6) at (6, 0) {};
    \node[label=below:$u_7$] (u7) at (7, 0) {};
    \node[label=below:$u_8$] (u8) at (8, 0) {};
    \node[label=below:$u_9$] (u9) at (9, 0) {};
    \node[label=below:$u_{10}$] (u10) at (10, 0) {};
    \draw (u1) -- (v1) -- (u2); 
    \draw (u3) -- (v2) -- (u4);
    \draw (v2) -- (u2);  \draw (v2) -- (u5); \draw (u5) -- (v3); 
    \draw (u5) -- (v4);  \draw (u5) -- (v6); \draw (u5) -- (v8); 
    \draw (v4) -- (u6);  \draw (v4) -- (u7); \draw (v5) -- (u6); 
    \draw (v5) -- (u7);  \draw (v6) -- (u8); \draw (v7) -- (u8); 
    \draw (v7) -- (u9);  \draw (v7) -- (u10); \draw (v8) -- (u9); 
    \draw (v8) -- (u10); \draw (v9) -- (u10); 
\end{tikzpicture}\\
\end{center}
Since $G$ contains two leaf vertices $u_4$ and $v_3$ of distance $3$, it follows from Lemma \ref{distance3} that
\[
\deg\ h_{R/I(G)}(t) < \alpha(G).
\]
Using \texttt{Macaulay2} \cite{GS}, we have $\deg\ h_{R/I(G)}(t) = 9$ and $\alpha(G) = 11$.
\end{Example}

\begin{Example}
  Let $G$ be the connected bipartite graph as follows.
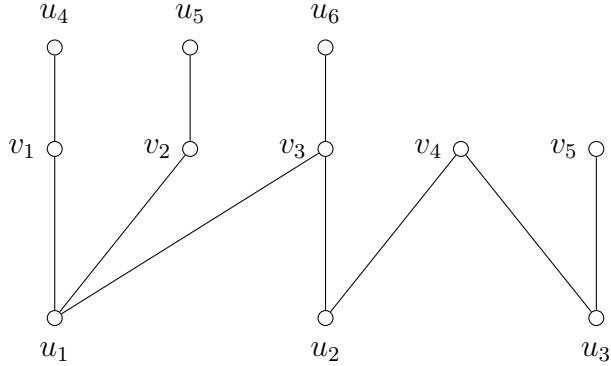
\begin{center}
  \begin{tikzpicture}[scale=0.9, every node/.style={draw=black, fill=white, circle, inner sep=2pt}]
    \node[label=above:$u_4$]  (u4) at (1, 4) {};
    \node[label=above:$u_5$]  (u5) at (3, 4) {}; 
    \node[label=above:$u_6$]  (u6) at (5, 4) {};
    \node[label=below:$u_1$]  (u1) at (1, 0) {};
    \node[label=below:$u_2$]  (u2) at (5, 0) {};
    \node[label=below:$u_3$]  (u3) at (9, 0) {}; 
    \node[label=left:$v_1$]  (v1) at (1, 2.5) {};
    \node[label=left:$v_2$]  (v2) at (3, 2.5) {};
    \node[label=left:$v_3$]  (v3) at (5, 2.5) {};
    \node[label=left:$v_4$]  (v4) at (7, 2.5) {};
    \node[label=left:$v_5$]  (v5) at (9, 2.5) {};
    \draw (u1) -- (v1) -- (u4);
    \draw (u1) -- (v2) -- (u5);
    \draw (u1) -- (v3) -- (u6);
    \draw (u2) -- (v3);
    \draw (u2) -- (v4);
    \draw (u3) -- (v4);
    \draw (u3) -- (v5);
\end{tikzpicture}\\
\end{center}
 We have $U_2(G) = \{u_1, u_2, v_4\}$. Using the reduction process to $G$, we remove $u_1$, $u_2$ and $v_4$, we get a subgraph that contains stars as follows.

\begin{center}
  \begin{tikzpicture}[scale=0.9, every node/.style={draw=black, fill=white, circle, inner sep=2pt},
ring/.style={circle, draw, fill=black, inner sep=1pt}]
    \node[label=above:$u_4$]  (u4) at (1, 4) {};
    \node[label=above:$u_5$]  (u5) at (3, 4) {}; 
    \node[label=above:$u_6$]  (u6) at (5, 4) {};
    \node[ring]  (u1) at (1, 0) {};
    \node[ring]  (u2) at (5, 0) {};
    \node[label=below:$u_3$]  (u3) at (9, 0) {}; 
    \node[label=left:$v_1$]  (v1) at (1, 2.5) {};
    \node[label=left:$v_2$]  (v2) at (3, 2.5) {};
    \node[label=left:$v_3$]  (v3) at (5, 2.5) {};
    \node[ring]  (v4) at (7, 2.5) {};
    \node[label=left:$v_5$]  (v5) at (9, 2.5) {};
    \draw (v1) -- (u4);
    \draw[dotted] (u1) -- (v1) (u1) -- (v2) (u1) -- (v3) (u2) -- (v3) (u2) -- (v4);
    \draw (v2) -- (u5);
    \draw (v3) -- (u6);
    \draw[dotted] (u3) -- (v4);
    \draw (u3) -- (v5);
\end{tikzpicture}
\end{center}
 By Theorem \ref{bipartite}, we get $$\deg \ h_{R/I(G)}(t) = \alpha(G) = 6.$$
\end{Example}
\begin{Example}
 Let $G$ be the bipartite graph as follows.
\begin{center}
  \begin{tikzpicture}[scale=0.9, every node/.style={draw=black, fill=white, circle, inner sep=2pt}]
    \node[label=below:$u_4$] (u4) at (7, 0) {};
    \node[label=above:$u_5$] (u5) at (1, 5) {}; 
    \node[label=below:$u_1$] (u1) at (1, 0) {};
    \node[label=below:$u_2$] (u2) at (3, 0) {};
    \node[label=below:$u_3$] (u3) at (5, 0) {}; 
    \node[label=left:$v_1$] (v1) at (1, 2.5) {};
    \node[label=left:$v_2$] (v2) at (3, 2.5) {};
    \node[label=left:$v_3$] (v3) at (5, 2.5) {};
    \node[label=left:$v_4$] (v4) at (7, 2.5) {};
    \draw (u1) -- (v1) -- (u5);
    \draw (u1) -- (v2);
    \draw (u2) -- (v2);
    \draw (u3) -- (v2);
    \draw (u4) -- (v2);
    \draw (u2) -- (v4);
    \draw (u3) -- (v3);
    \draw (u3) -- (v4);
    \draw (u4) -- (v4);
\end{tikzpicture}
\end{center}
We have $U_2(G) = \{u_1, v_2, v_4\}$. 
Applying the reduction process to $G$, we obtain the subgraph
which consists of stars and isolated vertices as follows. 
\begin{center}
  \begin{tikzpicture}[scale=0.9, every node/.style={draw=black, fill=white, circle, inner sep=2pt}, ring/.style={circle, draw, fill=black, inner sep=1pt}]
  \node[label=below:$u_4$] (u4) at (7, 0) {};
    \node[label=above:$u_5$] (u5) at (1, 5) {}; 
    \node[ring] (u1) at (1, 0) {};
    \node[label=below:$u_2$] (u2) at (3, 0) {};
    \node[label=below:$u_3$] (u3) at (5, 0) {}; 
    \node[label=left:$v_1$] (v1) at (1, 2.5) {};
    \node[ring] (v2) at (3, 2.5) {};
    \node[label=left:$v_3$] (v3) at (5, 2.5) {};
    \node[ring] (v4) at (7, 2.5) {};
    \draw[dotted] (u1) -- (v1);
    \draw (v1)-- (u5);
    \draw[dotted] (u1) -- (v2);
    \draw[dotted]  (u2) -- (v2);
    \draw[dotted]  (u3) -- (v2);
    \draw[dotted]  (u4) -- (v2);
    \draw[dotted]  (u2) -- (v4);
    \draw (u3) -- (v3);
    \draw[dotted]  (u3) -- (v4);
    \draw[dotted]  (u4) -- (v4);
    
\end{tikzpicture}
\end{center}
Therefore, by Theorem \ref{bipartite},
$$
\deg\ h_{R/I(G)}(t) < \alpha(G).
$$
Using \texttt{Macaulay2} \cite{GS}, we have $\deg\ h_{R/I(G)}(t) = 4$ and $\alpha(G) = 5$.
\end{Example}

As an application of Theorem~\ref{general}, we consider rooted trees.
For completeness, we recall the following definition.
\begin{Definition}
  A \emph{rooted tree} $T$ is a tree in which one vertex is designated as the root, and every edge is directed away from the root. If $u$ and $v$ are vertices in the rooted tree $T$, then $v$ is called a \emph{child} of $u$ if there is a directed edge from $u$ to $v$. A vertex with no children is called a \emph{leaf}, while vertices that have children are called \emph{internal vertices}. The \emph{level} of a vertex $v$ in a rooted tree $T$ is the length of the unique path from the root to $v$.
A rooted tree $T$ is said to be a \emph{$m$-ary tree} if every internal vertex has no more than $m$ children.
\end{Definition}

\begin{center}
  \begin{tikzpicture}[scale=0.9, every node/.style={draw=black, fill=white, circle, inner sep=2pt}]
    
    \node (v1) at (0, 5) {};
    \node (v2) at (-1, 4) {};
    \node (v3) at (0, 4) {};
    \node (v4) at (1, 4) {};
    \node (v5) at (-2, 3) {}; 
    \node (v6) at (-1, 3) {};
    \node (v7) at (0, 3) {}; 
    \node (v8) at (1, 3) {};
    \node (v9) at (2, 3) {};
    \draw (v1) -- (v2);
    \draw (v1) -- (v3);
    \draw (v1) -- (v4);
    \draw (v2) -- (v5);
    \draw (v2) -- (v6);
    \draw (v3) -- (v7);
    \draw (v4) -- (v8);
    \draw (v4) -- (v9);
\end{tikzpicture}\\
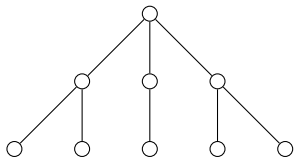
\textbf{Figure 1.} The $3$-ary tree with all leaves of level $2$.
\end{center}
\begin{Corollary}\label{rootedtree}
  Let $T$ be an $m$-ary tree with all leaves of level $n$, where $m \geq 1$ and $n \geq 0$. Then $\deg\ h_{R/I(T)}(t) = \alpha(T)$ if and only if $n \equiv 1,2 \pmod 3$.  
\end{Corollary}
\begin{proof}
  Let $T$ be an $m$-ary tree rooted at $v_1$ in which all leaves have level $n$. 
We consider the following three cases.\\
\textbf{Case 1:} If $n \equiv 0 \pmod{3}$, applying the elimination process recursively, we remove the vertices at levels $n - 2 - 3i$ for $i = 0, \ldots, \dfrac{n-3}{3}$. As a result, we obtain a subgraph consisting of stars and an isolated vertex.\\
\textbf{Case 2:} If $n \equiv 1 \pmod{3}$, applying the elimination process recursively, we remove the vertices at levels $n - 2 - 3i$ for $i = 0, \ldots, \dfrac{n-4}{3}$. As a result, we obtain a subgraph consisting of stars.\\
\textbf{Case 3:} If $n \equiv 2 \pmod{3}$, applying the elimination process recursively, we remove the vertices at levels $n - 2 - 3i$ for $i = 0, \ldots, \dfrac{n-2}{3}$. As a result, we obtain a subgraph consisting of stars. \\
According to Theorem~\ref{general}, it follows that $n \equiv 1,2 \pmod{3}$ if and only if $\deg\ h_{R/I(T)}(t) = \alpha(T)$.
\end{proof}
From Corollary \ref{rootedtree}, we obtain the following corollary.
\begin{Corollary}
 Let $T$ be an $m$-ary tree with all leaves of level $n$. Then $a(R/I(T)) = 0$ if and only if $n \equiv 1,2 \pmod 3$.  
\end{Corollary}
As another application of Theorem~\ref{general}, we consider bipartite graphs in which every vertex of one part is whiskered.
\begin{Corollary}
Let $G$ be a connected bipartite graph with partition $(U,V)$. If every vertex in $V$ is whiskered, then $\deg\ h_{R/I(G)}(t) = \alpha(G)$.

\end{Corollary}
\begin{proof}
 Suppose that $U_L$ is the set of leaves in $U$ and $\overline{U}_L$ the complement of $U_L$ in $U$. Since every vertex in $V$ is whiskered, we have $\overline{U}_L = U_2(G)$. Applying the elimination process to $G$, we remove the vertices in $U_2(G)$ and obtain a subgraph that consists of stars. According to Theorem \ref{bipartite}, we have $\deg\ h_{R/I(G)}(t) = \alpha(G)$.
\end{proof}
Motivated by Theorem \ref{bipartite}, we pose the following question.
\begin{Question}
Let $G$ be a connected bipartite graph in which every vertex has degree at least $2$.  
When does $I(G,-1) \neq 0$?
\end{Question}
We give an answer to this question.    
For this purpose, we introduce the following setup.

{\bf Setup:} Let $G$ be a bipartite graph with bipartition $(U,V)$ such that every vertex has degree at least $2$, where $|U|=m$ and $|V|=n$. 
Then $G$ can be obtained from the complete bipartite graph $K_{m,n}$ by deleting a finite set of edges.

Let $a_1,a_2,\ldots,a_k$ denote the edges removed from $K_{m,n}$, where each
$a_i = u_i v_i$ with $u_i \in U$ and $v_i \in V$.
For $i=1,\ldots,k$, define $G_i := G_{i-1} - a_i$, with $G_0 := K_{m,n}$ and $G_k = G$.

In the following theorem, we give a necessary and sufficient condition under which the degree of the $h$-polynomial of $G$ attains its maximum value.
\begin{Theorem}\label{degreeatleast2}
Under Setup, we have
$$\deg \ h_{R/I(G)}(t) = \alpha(G) \text{ if and only if } \sum_{i=1}^{k}I\left(G_{i-1}-(N_{G_{i-1}}(u_i)\cup N_{G_{i-1}}(v_i)),-1 \right)\neq 1.$$
\end{Theorem}
\begin{proof}
By Lemma \ref{Recusive} (ii), we have
$$I(G_0,t) = I(G_1,t) - t^2I(G_0-(N_{G_0}(u_1)\cup N_{G_0}(v_1)),t).$$
It follows that
$$I(G_1,t) = I(G_0,t) + t^2I(G_0-(N_{G_0}(u_1)\cup N_{G_0}(v_1)),t). $$
Proceeding recursively, we obtain
\begin{equation}\label{plus}
  I(G_i,t) = I(G_{i-1},t) + t^2I(G_{i-1}-(N_{G_{i-1}}(u_i)\cup N_{G_{i-1}}(v_i)),t)
\end{equation}
for $i=1,\ldots,k$. 
Summing the equalities \eqref{plus} for $i=1,\ldots,k$, we get
\begin{equation}\label{plus1}
  I(G,t) = I(G_k,t) = I(G_0,t) + t^2\sum_{i=1}^{k} I(G_{i-1}-(N_{G_{i-1}}(u_i)\cup N_{G_{i-1}}(v_i)),t).
\end{equation}
Since $I(G_0,t) = I(K_{m,n},t) = (1+t)^m + (1+t)^n - 1,$ we have $I(G_0,-1) = -1$. By substituting $t=-1$ into \eqref{plus1}, we obtain
$$I(G,-1) = \sum_{i=1}^{k} I(G_{i-1}-(N_{G_{i-1}}(u_i)\cup N_{G_{i-1}}(v_i)),-1) - 1.$$
Therefore, 
$$I(G,-1) \neq 0 \quad\text{if and only if}\quad \sum_{i=1}^{k} I(G_{i-1}-(N_{G_{i-1}}(u_i)\cup N_{G_{i-1}}(v_i)),-1) \neq 1.$$
 From Theorem \ref{chacracteristic}, the result follows.
\end{proof}
In the following corollary, we describe several choices of edges $a_1,\ldots,a_k$ in Setup that guarantee the $h$-polynomial of $R/I(G)$ attains its maximal degree.
\begin{Corollary}\label{degree21}
With Setup, if the edges $a_1,\ldots,a_k$ satisfy one of the following conditions:
\begin{itemize}
  \item[(i)] $a_1,\ldots,a_k$ form a matching of $K_{m,n}$ with $k \neq 1$;
  \item[(ii)] $a_1,\ldots,a_k$ form the edges of an even cycle with $k > 4$.
\end{itemize}
Then
\[
\deg\ h_{R/I(G)}(t) = \alpha(G).
\]
\end{Corollary}
\begin{proof}
We note that the independence polynomial of $G$ does not depend on the order in which the edges $a_1,\ldots,a_k$ are removed.

\medskip
\noindent
{\rm (i)} Suppose that $a_1,\ldots,a_k$ form a matching of $K_{m,n}$.
We remove the edges $a_1,\ldots,a_k$ successively from $K_{m,n}$ to obtain the graph $G$.
Since $a_1,\ldots,a_k$ form a matching, we have
 $G_{i-1}-(N_{G_{i-1}}(u_i)\cup N_{G_{i-1}}(v_i))= \emptyset$ for $i=1,\ldots,k$ and hence  $I(G_{i-1}-(N_{G_{i-1}}(u_i)\cup N_{G_{i-1}}(v_i)),t) = 1$ for $i=1,\ldots,k$. Therefore,
$$\sum_{i=1}^{k} I(G_{i-1}-(N_{G_{i-1}}(u_i)\cup N_{G_{i-1}}(v_i)),-1) = k.$$
Since $k\neq 1$, it follows from Theorem~\ref{degreeatleast2} that $$\deg\ h_{R/I(G)(t)} = \alpha(G).$$
(ii) Let $$a_1=u_1v_1, a_2=u_2v_2, a_3=u_3v_3,\ldots, a_{k}=u_kv_k$$
such that $$u_1 = u_2, u_3=u_4,\ldots,u_{k-1}=u_k$$ and $$v_1=v_k,v_2=v_3,\ldots,v_{k-2}=v_{k-1}.$$
Then 
$K_{m,n}-(N_{K_{m,n}}(u_1)\cup N_{K_{m,n}}(v_1)) = \emptyset$ and 
$G_{i-1}-(N_{G_{i-1}}(u_i)\cup N_{G_{i-1}}(v_i)$ is an isolated vertex for $ i=2,\ldots,k-1$ 
and 
$G_{k-1} - (N_{G_{k-1}}(u_k) \cup N_{G_{k-1}}(v_k))$ is an edge. 
Therefore, 
 $$\sum_{i=1}^{k}I(G_{i-1}-(N_{G_{i-1}}(u_i)\cup N_{G_{i-1}}(v_i)),t) = 1 + (k-2)(1+t) + (1+2t). $$
It follows that 
$$\sum_{i=1}^{k}I(G_{i-1}-(N_{G_{i-1}}(u_i)\cup N_{G_{i-1}}(v_i)),-1) = 0 \neq 1$$ 
Hence, by Theorem~\ref{degreeatleast2},
$$\deg\ h_{R/I(G)}(t) = \alpha(G).$$
\end{proof}
\begin{Example} Let $G$ be the bipartite graph:
\begin{center}
  \begin{tikzpicture}[scale=0.9, dot/.style={circle, draw, fill=white, inner sep=2pt}]
 \node[draw=none, fill=none] at (-1, 1.5) {$G:$};
    \node[dot,label=below:$u_1$] (u1) at (0, 0) {};
    \node[dot,label=below:$u_2$] (u2) at (2, 0) {}; 
    \node[dot,label=below:$u_4$] (u4) at (4, 0) {};
    \node[dot,label=below:$u_3$] (u3) at (6, 0) {}; 
    \node[dot,label=below:$u_5$] (u5) at (8, 0) {}; 

    \node[dot,label=above:$v_1$] (v1) at (1, 3) {};
    \node[dot,label=above:$v_2$] (v2) at (3, 3) {};
    \node[dot,label=above:$v_3$] (v3) at (5, 3) {};
    \node[dot,label=above:$v_4$] (v4) at (7, 3) {};
   
    \draw (v1) -- (u1)  (v1) -- (u3)  (v1) -- (u4)  (v1) -- (u5);

    \draw  (v2) -- (u2)  (v2) -- (u3)  (v2) -- (u4)  (v2) -- (u5);
    \draw (v3) -- (u1)  (v3) -- (u2)  (v3) -- (u3)  (v3) -- (u5);
  \draw (v4) -- (u1)  (v4) -- (u2)  (v4) -- (u3)  (v4) -- (u4)  (v4) -- (u5);
\end{tikzpicture}
\end{center}
We observe that $G$ is obtained from $K_{4,5}$ by deleting the edges
$v_1u_2$, $v_2u_1$, and $v_3u_3$, as follows.
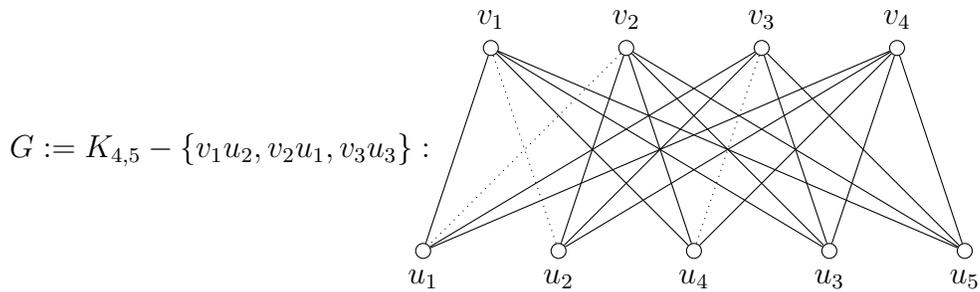
\begin{center}
  \begin{tikzpicture}[scale=0.9, dot/.style={circle, draw, fill=white, inner sep=2pt}]
 \node[draw=none, fill=none] at (-3, 1.5) {$G:= K_{4,5}-\{v_1u_2,v_2u_1,v_3u_3\}:$};
    \node[dot,label=below:$u_1$] (u1) at (0, 0) {};
    \node[dot,label=below:$u_2$] (u2) at (2, 0) {}; 
    \node[dot,label=below:$u_4$] (u4) at (4, 0) {};
    \node[dot,label=below:$u_3$] (u3) at (6, 0) {}; 
    \node[dot,label=below:$u_5$] (u5) at (8, 0) {}; 

    \node[dot,label=above:$v_1$] (v1) at (1, 3) {};
    \node[dot,label=above:$v_2$] (v2) at (3, 3) {};
    \node[dot,label=above:$v_3$] (v3) at (5, 3) {};
    \node[dot,label=above:$v_4$] (v4) at (7, 3) {};
   
    \draw (v1) -- (u1)    (v1) -- (u3)  (v1) -- (u4)  (v1) -- (u5);
    \draw [dotted] (v1) -- (u2);
    \draw (v2) -- (u2)  (v2) -- (u3)  (v2) -- (u4)  (v2) -- (u5);
    \draw [dotted] (v2) -- (u1);
    \draw (v3) -- (u1)  (v3) -- (u2)  (v3) -- (u3) (v3) -- (u5);
\draw [dotted] (v3) -- (u4);
  \draw (v4) -- (u1)  (v4) -- (u2)  (v4) -- (u3)  (v4) -- (u4)  (v4) -- (u5);
\end{tikzpicture}
\end{center}
Since $v_1u_2$, $v_2u_1$, and $v_3u_3$ form a matching of $K_{4,5}$, by Corollary~\ref{degree21} we have
$$\deg\ h_{R/I(G)}(t) = \alpha(G).$$
\end{Example}

\begin{Example} Let $G$ be the bipartite graph:
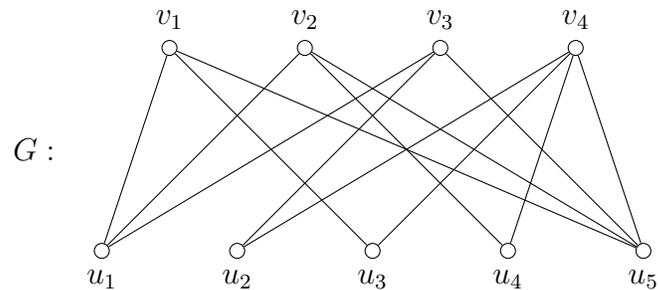
\begin{center}
  \begin{tikzpicture}[scale=0.9, dot/.style={circle, draw, fill=white, inner sep=2pt}]
 \node[draw=none, fill=none] at (-1, 1.5) {$G:$};
    \node[dot,label=below:$u_1$] (u1) at (0, 0) {};
    \node[dot,label=below:$u_2$] (u2) at (2, 0) {}; 
    \node[dot,label=below:$u_3$] (u3) at (4, 0) {};
    \node[dot,label=below:$u_4$] (u4) at (6, 0) {}; 
    \node[dot,label=below:$u_5$] (u5) at (8, 0) {}; 

    \node[dot,label=above:$v_1$] (v1) at (1, 3) {};
    \node[dot,label=above:$v_2$] (v2) at (3, 3) {};
    \node[dot,label=above:$v_3$] (v3) at (5, 3) {};
    \node[dot,label=above:$v_4$] (v4) at (7, 3) {};
   
    \draw (v1) -- (u1) 
(v1) -- (u3)  
(v1) -- (u5);
    \draw (v2) -- (u1)  
(v2) -- (u4)  
(v2) -- (u5);
    \draw (v3) -- (u1)  (v3) -- (u2)  
 (v3) -- (u5);
    \draw 
(v4) -- (u2)  (v4) -- (u3)  (v4) -- (u4)  (v4) -- (u5);
\end{tikzpicture}
\end{center}
We observe that $G$ is obtained from $K_{4,5}$ by deleting the edges
$v_1u_2$, $u_2v_2$, $v_2u_3$, $u_3v_3$, $v_3u_4,$ and  $u_4v_1$ as follows.
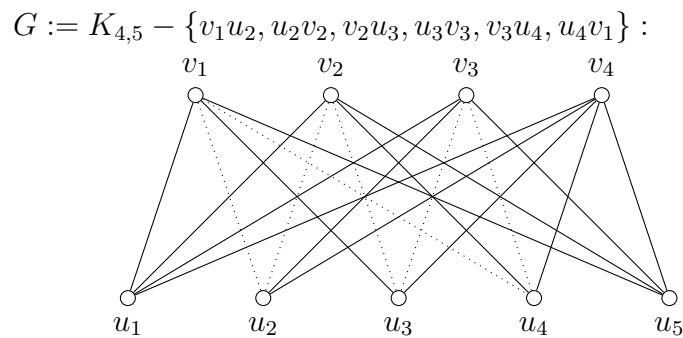
\begin{center}
  \begin{tikzpicture}[scale=0.9, dot/.style={circle, draw, fill=white, inner sep=2pt}]
 \node[draw=none, fill=none] at (3, 4) {$G:= K_{4,5}-\{v_1u_2, u_2v_2, v_2u_3, u_3v_3, v_3u_4, u_4v_1\}:$};
    \node[dot,label=below:$u_1$] (u1) at (0, 0) {};
    \node[dot,label=below:$u_2$] (u2) at (2, 0) {}; 
    \node[dot,label=below:$u_3$] (u3) at (4, 0) {};
    \node[dot,label=below:$u_4$] (u4) at (6, 0) {}; 
    \node[dot,label=below:$u_5$] (u5) at (8, 0) {}; 

    \node[dot,label=above:$v_1$] (v1) at (1, 3) {};
    \node[dot,label=above:$v_2$] (v2) at (3, 3) {};
    \node[dot,label=above:$v_3$] (v3) at (5, 3) {};
    \node[dot,label=above:$v_4$] (v4) at (7, 3) {};
   
    \draw (v1) -- (u1)    (v1) -- (u3)   (v1) -- (u5);
    \draw [dotted] (v1) -- (u2) ;
    \draw (v2) -- (u1)   (v2) -- (u4)  (v2) -- (u5);
    \draw [dotted] (v2) -- (u2)  (v2) -- (u3);
    \draw (v3) -- (u1)  (v3) -- (u2) (v3) -- (u5);
\draw [dotted] (v3) -- (u3) (v3) -- (u4);
  \draw (v4) -- (u1) (v4) -- (u2)  (v4) -- (u3)  (v4) -- (u4)  (v4) -- (u5);
    \draw [dotted] (v1) -- (u4);
\end{tikzpicture}
\end{center}
Since $v_1u_2, u_2v_2, v_2u_3, u_3v_3,$ $ v_3u_4,$ and  $u_4v_1$ form a cycle of length $6$ of $K_{4,5}$, by Corollary \ref{degree21} we have 
$$\deg\ h_{R/I(G)}(t) = \alpha(G).$$
\end{Example}

\begin{Corollary}
Let $G$ be a connected bipartite graph.
Assume that, using the elimination process, $G$ decomposes into a disjoint union of star graphs
$S_1,\ldots,S_r$ and bipartite graphs $G_1,\ldots,G_s$, where every vertex of each $G_j$ has degree at least $2$.
If each $G_j$ either satisfies the conditions of Corollary~\ref{degree21} or is a complete bipartite graph, then
\[
\deg\ h_{R/I(G)}=\alpha(G).
\]
\end{Corollary}
\begin{proof}
By Theorem~\ref{general}, we have $\deg\ h_{R/I(G)}=\alpha(G)$ if and only if
$I(G_j,-1) \neq 0$ for all $j=1,\ldots,s$.
If $G_j$ satisfies the conditions of Corollary~\ref{degree21}, then
$I(G_j,-1)\neq0$.
If $G_j$ is a complete bipartite graph, then
$I(G_j,-1)=I(K_{m,n},-1)=-1\neq0$.
Therefore $I(G_j,-1)\neq0$ for all $j$, and the result follows.
\end{proof}
\subsection{Degree of $h$-polynomials of Cameron-Walker graphs}
In this section, we investigate the degree of the $h$-polynomial of Cameron-Walker graphs in terms of the independence number. 

It is well known that given graph $G$ 
$$\nu(G) \leq \reg R/I(G) \leq \mu(G),$$
where $\nu(G)$ is the induced matching number of $G$ and $\mu(G)$ is the matching number of $G$. A graph $G$ with $\nu(G) = \mu(G)$ is called a Cameron-Walker graph. Cameron-Walker graphs are classified in \cite[Theorem 1]{Cameron}, as follows.
\begin{Theorem}
  Let $G$ be a connected graph. Then $G$ is Cameron-Walker graph if and only if it is one of the following graphs:\\
{\rm (i)} A star graph;\\
{\rm (ii)} A star triangle;\\
{\rm (iii)} A graph consisting of a connected bipartite graph with a partition $(U,V)$ such that there at least one leaf edge is attached to each vertex $u\in U$ and that there may be possible some pendant triangles attached to each vertex $v \in V$.
\end{Theorem}
\begin{center}
\begin{tikzpicture}[
    dot/.style={circle, draw, fill=white, inner sep=1.2pt},
    every label/.style={font=\scriptsize}
]
    \draw[dashed] (-1, -1.5) rectangle (10, 1.5);
    \node at (4.5, 0) {connected bipartite graph on $\{v_1, \dots, v_m\} \cup \{u_1, \dots, u_n\}$};
    \node[dot, label=below:$v_1$] (v1) at (0.5, 1.5) {};
    \node[dot, label=below:$v_2$] (v2) at (3.5, 1.5) {};
    \node at (5.25, 0.8) {$\dots$};
    \node[dot, label=below:$v_m$] (vm) at (7, 1.5) {};
    \node[dot, label=above:$u_1$] (w1) at (0.2, -1.5) {};
    \node[dot, label=above:$u_2$] (w2) at (3.5, -1.5) {};
    \node at (5.25, -0.8) {$\dots$};
    \node[dot, label=above:$u_n$] (wn) at (7, -1.5) {};
    \node[dot, label=above:$x_1^{(1)}$] (x11) at (-0.2, 2.5) {};
    \node[dot, label=above:$x_{s_1}^{(1)}$] (x1s) at (1.2, 2.5) {};
    \draw (v1) -- (x11); \draw (v1) -- (x1s);
    \node at (0.5, 2.3) {$\dots$};
    \node[dot, label=above:$x_1^{(2)}$] (x21) at (2.8, 2.5) {};
    \node[dot, label=above:$x_{s_2}^{(2)}$] (x2s) at (4.2, 2.5) {};
    \draw (v2) -- (x21); \draw (v2) -- (x2s);
    \node at (3.5, 2.3) {$\dots$};
    \node[dot, label=above:$x_1^{(m)}$] (xm1) at (6.3, 2.5) {};
    \node[dot, label=above:$x_{s_m}^{(m)}$] (xms) at (7.7, 2.5) {};
    \draw (vm) -- (xm1); \draw (vm) -- (xms);
    \node at (7, 2.3) {$\dots$};
    \node[dot, label=left:$y_{1,1}^{(1)}$] (y111) at (-0.5, -2.8) {};
    \node[dot, label=below:$y_{1,2}^{(1)}$] (y112) at (-0.1, -3.2) {};
    \node[dot, label=below:$y_{t_1,1}^{(1)}$] (y1t1) at (0.6, -3.2) {};
    \node[dot, label=right:$y_{t_1,2}^{(1)}$] (y1t2) at (1.0, -2.8) {};
    \draw (w1) -- (y111) -- (y112) -- (w1);
    \draw (w1) -- (y1t1) -- (y1t2) -- (w1);
    \node at (0.3, -2.8) {$\dots$};
    \node[dot, label=left:$y_{1,1}^{(2)}$] (y211) at (2.7, -2.8) {};
    \node[dot, label=below:$y_{1,2}^{(2)}$] (y212) at (3.1, -3.2) {};
    \node[dot, label=below:$y_{t_2,1}^{(2)}$] (y2t1) at (3.9, -3.2) {};
    \node[dot, label=right:$y_{t_2,2}^{(2)}$] (y2t2) at (4.3, -2.8) {};
    \draw (w2) -- (y211) -- (y212) -- (w2);
    \draw (w2) -- (y2t1) -- (y2t2) -- (w2);
    \node at (3.5, -2.8) {$\dots$};
    \node[dot, label=left:$y_{1,1}^{(n)}$] (yn11) at (6.2, -2.8) {};
    \node[dot, label=below:$y_{1,2}^{(n)}$] (yn12) at (6.6, -3.2) {};
    \node[dot, label=below:$y_{t_n,1}^{(n)}$] (ynt1) at (7.4, -3.2) {};
    \node[dot, label=right:$y_{t_n,2}^{(n)}$] (ynt2) at (7.8, -2.8) {};
    \draw (wn) -- (yn11) -- (yn12) -- (wn);
    \draw (wn) -- (ynt1) -- (ynt2) -- (wn);
    \node at (7, -2.8) {$\dots$};
    \node at (5.25, 2) {$\dots$};
    \node at (5.25, -2.3) {$\dots$};
\end{tikzpicture}\\
\textbf{Figure 2.} Cameron-Walker graph.
\end{center}
\begin{Theorem}\label{Cameron}
Let $G$ be a Cameron--Walker graph. Then
\[
\deg h_{R/I(G)}=\alpha(G)
\]
unless $G$ is a star triangle with an even number of triangles. 
In that case,
\[
\deg h_{R/I(G)}=\alpha(G)-1.
\]
\end{Theorem}
\begin{proof}
  (i) Suppose that $G$ is a star triangle with $m\ge 1$ triangles and center vertex $u$.  By Lemma \ref{Recusive} (i), we get
$$I(G,t) = I(G-u,t) + tI(G-N[u],t).$$
But $G-N[u]$ is the empty graph, so $I(G - N[u],t) =1$. The graph $G - u$ consists of a disjoint union of $m$ copies of $K_2$. Therefore, $I(G-u,t)=(1+2t)^m$. Thus, 
$$I(G,t) = (1+2t)^m + t.$$
Therefore,
$$
I(G,-1) = \left\{
            \begin{array}{ll}
              -2 & \hbox{ if } m \text{ is odd}; \\
              0 & \hbox{ if } m \text{ is even}.
            \end{array}
          \right.
$$
If $m$ is even, then $$I'(G,t) = 2m(1+2t)^{m-1} + 1,$$ so $I'(G,-1) \neq 0$. According to Theorem \ref{chacracteristic}, we obtain
$$
\deg\ h_{R/I(G)}(t) = \left\{
                        \begin{array}{ll}
                          \alpha(G) & \hbox{ if }  m \text{ is odd}; \\
                          \alpha(G) - 1 & \hbox{ if } m \text{ is even}.
                        \end{array}
                      \right.
$$

(ii). {\bf Case 1.} If $G$ is a star graph, then the conclusion follows from Lemma \ref{star}.
\\
{\bf Case 2.} In this case, we can apply Theorem \ref{general} to Cameron--Walker graphs. 
However, for the reader’s convenience, we provide a simple and detailed proof.
 If $G$ is neither a star graph nor a star triangle graph, suppose that $(U,V)$ is a partition of the connected bipartite graph of $G$ and that every $v\in V$ has at least one leaf edge. There may be some pendant triangles attached to each vertex $v \in V$. Assume that $V = \{v_1, \ldots, v_n\}$ and that $\deg(v_i) \geq 2$ for $i = 1, \ldots, r$. Set $G_i = G_{i-1}-v_i$ for $i=1,\ldots,r$, where $G_0:=G$. 
By Lemma \ref{Recusive} (i), we have
$$I(G_{i-1},t) = I(G_i,t) + tI(G_{i-1}- N_{G_{i-1}}[v_i],t)$$
for $i=1,\ldots,r$. Adding the above equalities, we get
$$I(G,t) = I(G_r,t) + t \sum_{i=1}^{r}I(G_{i-1}-N_{G_{i-1}}[v_i],t).$$ 
We note that $G_{i-1}-N_{G_{i-1}}[v_i]$ contains at least one isolated vertex for $i =1,\ldots,r$, therefore $$\sum_{i=1}^{r}I(G_{i-1}-N_{G_{i-1}}[v_i],-1) = 0.$$
Moreover, $G_r = G-\{v_1,\ldots,v_r\}$ is a graph that consists of a disjoint union of star graphs. It follows that
$$I(G_r,-1) \neq 0.$$
Hence,
$$I(G,-1) \neq 0.$$
By Theorem \ref{chacracteristic}, we obtain $$\deg\ h_{R/I(G)}(t) = \alpha(G).$$
\end{proof}

We give an example of Cameron-Walker graph which would be helpful to under stand of the proof of Theorem \ref{Cameron}.
\begin{Example}
Let $G$ be the following Cameron-Walker graph:\\
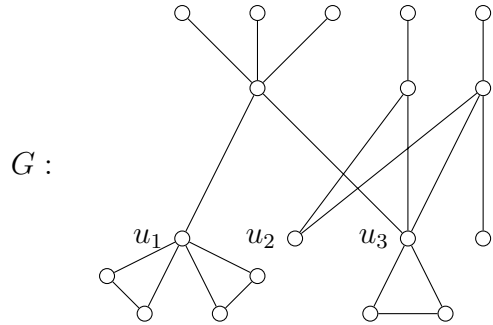
  \begin{center}
\begin{tikzpicture}[
    dot/.style={circle, draw, fill=white, inner sep=2pt}
]


        \node[dot] (t1) at (-1,3) {};
        \node[dot] (t2) at (0,3) {};
        \node[dot] (t3) at (1,3) {};
        \node[dot] (t4) at (2,3) {};
        \node[dot] (t5) at (3,3) {};
        
        \node[dot] (m1) at (0,2) {};
        \node[dot] (m2) at (2,2) {};
        \node[dot] (m3) at (3,2) {};
        \draw (m1)--(t1) (m1)--(t2) (m1)--(t3);
        \draw (m2)--(t4) (m3)--(t5);
        
        \node[dot,label=left:$u_1$] (k1) at (-1,0) {};
        \node[dot,label=left:$u_2$] (k2) at (0.5,0) {};
        \node[dot,label=left:$u_3$] (k3) at (2,0) {};
        \node[dot] (k4) at (3,0) {};
        \draw (m1)--(k1) (m1)--(k3) (m2)--(k3) (m3)--(k2) (m3)--(k4) (m2)--(k2) (m3)--(k3);
 \node[draw=none, fill=none] at (-3, 1) {$G:$};
        \node[dot] (d1) at (-2,-0.5) {};
        \node[dot] (d2) at (-1.5,-1) {};
        \node[dot] (d3) at (-0.5,-1) {};
        \node[dot] (d4) at (0,-0.5) {};
        \draw (k1)--(d1) (k1)--(d2) (k1)--(d3) (k1)--(d4);
        \draw (d1)--(d2) (d3)--(d4);
        
        \node[dot] (e1) at (1.5,-1) {};
        \node[dot] (e2) at (2.5,-1) {};
        \draw (k3)--(e1) (k3)--(e2) (e1)--(e2);
\end{tikzpicture}
\end{center}
The induced subgraph graph $G-u_1$ is as follows. 
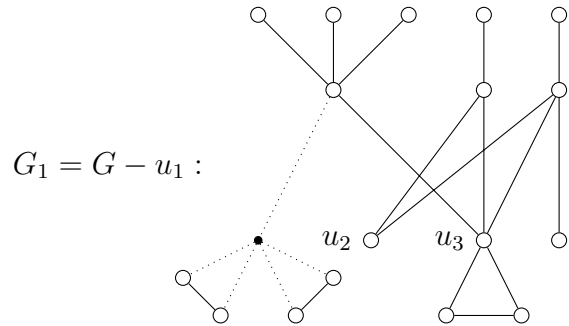
\begin{center}
\begin{tikzpicture}[
    dot/.style={circle, draw, fill=white, inner sep=2pt},
ring/.style={circle, draw, fill=black, inner sep=1pt}
]


        \node[dot] (t1) at (-1,3) {};
        \node[dot] (t2) at (0,3) {};
        \node[dot] (t3) at (1,3) {};
        \node[dot] (t4) at (2,3) {};
        \node[dot] (t5) at (3,3) {};
        
        \node[dot] (m1) at (0,2) {};
        \node[dot] (m2) at (2,2) {};
        \node[dot] (m3) at (3,2) {};
        \draw (m1)--(t1) (m1)--(t2) (m1)--(t3);
        \draw (m2)--(t4) (m3)--(t5);
        
        \node[ring] (k1) at (-1,0) {};
         \node[dot,label=left:$u_2$] (k2) at (0.5,0) {};
        \node[dot,label=left:$u_3$] (k3) at (2,0) {};
        \node[dot] (k4) at (3,0) {};
        \draw [dotted](m1)--(k1); 
        \draw (m1)--(k3) (m2)--(k3) (m3)--(k2) (m3)--(k4) (m2)--(k2) (m3)--(k3);
 \node[draw=none, fill=none] at (-3, 1) {$G_1=G-u_1:$};
        \node[dot] (d1) at (-2,-0.5) {};
        \node[dot] (d2) at (-1.5,-1) {};
        \node[dot] (d3) at (-0.5,-1) {};
        \node[dot] (d4) at (0,-0.5) {};
        \draw [dotted] (k1)--(d1) (k1)--(d2) (k1)--(d3) (k1)--(d4);
        \draw (d1)--(d2) (d3)--(d4);
        
        \node[dot] (e1) at (1.5,-1) {};
        \node[dot] (e2) at (2.5,-1) {};
        \draw (k3)--(e1) (k3)--(e2) (e1)--(e2);
\end{tikzpicture}
\end{center}
Also, the induced subgraph $G - N[u_1]$ consists of isolated vetices.
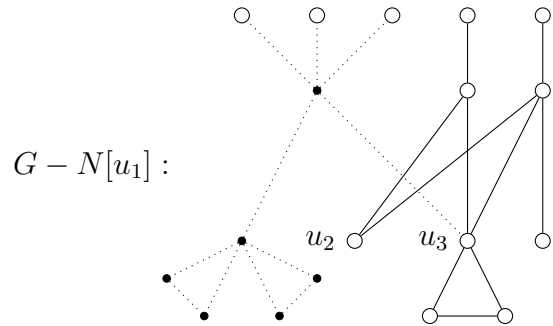
\begin{center}
\begin{tikzpicture}[
    dot/.style={circle, draw, fill=white, inner sep=2pt},
ring/.style={circle, draw, fill=black, inner sep=1pt}
]


        \node[dot] (t1) at (-1,3) {};
        \node[dot] (t2) at (0,3) {};
        \node[dot] (t3) at (1,3) {};
        \node[dot] (t4) at (2,3) {};
        \node[dot] (t5) at (3,3) {};
        
        \node[ring] (m1) at (0,2) {};
        \node[dot] (m2) at (2,2) {};
        \node[dot] (m3) at (3,2) {};
        \draw [dotted] (m1)--(t1) (m1)--(t2) (m1)--(t3);
        \draw (m2)--(t4) (m3)--(t5);
        
        \node[ring] (k1) at (-1,0) {};
          \node[dot,label=left:$u_2$] (k2) at (0.5,0) {};
        \node[dot,label=left:$u_3$] (k3) at (2,0) {};
        \node[dot] (k4) at (3,0) {};
        \draw [dotted](m1)--(k1); 
        \draw [dotted] (m1)--(k3);
        \draw  (m2)--(k3) (m3)--(k2) (m3)--(k4) (m2)--(k2) (m3)--(k3);
 \node[draw=none, fill=none] at (-3, 1) {$G-N[u_1]:$};
        \node[ring] (d1) at (-2,-0.5) {};
        \node[ring] (d2) at (-1.5,-1) {};
        \node[ring] (d3) at (-0.5,-1) {};
        \node[ring] (d4) at (0,-0.5) {};
        \draw [dotted] (k1)--(d1) (k1)--(d2) (k1)--(d3) (k1)--(d4);
        \draw [dotted] (d1)--(d2) (d3)--(d4);
        
        \node[dot] (e1) at (1.5,-1) {};
        \node[dot] (e2) at (2.5,-1) {};
        \draw  (k3)--(e1) (k3)--(e2) (e1)--(e2);
\end{tikzpicture}
\end{center}
The induced subgraph $G_3$ is a disjoint union of star graphs.
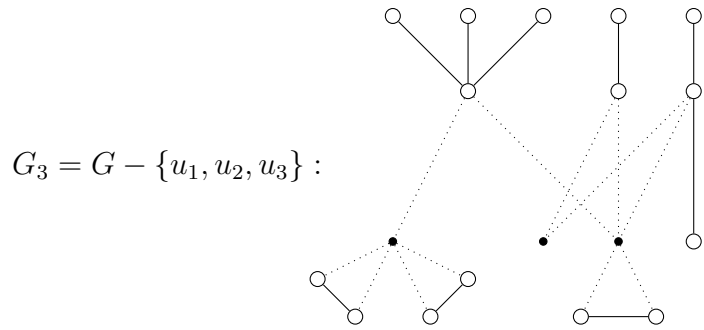
\begin{center}
\begin{tikzpicture}[
    dot/.style={circle, draw, fill=white, inner sep=2pt},
ring/.style={circle, draw, fill=black, inner sep=1pt}
]


        \node[dot] (t1) at (-1,3) {};
        \node[dot] (t2) at (0,3) {};
        \node[dot] (t3) at (1,3) {};
        \node[dot] (t4) at (2,3) {};
        \node[dot] (t5) at (3,3) {};
        
        \node[dot] (m1) at (0,2) {};
        \node[dot] (m2) at (2,2) {};
        \node[dot] (m3) at (3,2) {};
        \draw (m1)--(t1) (m1)--(t2) (m1)--(t3);
        \draw (m2)--(t4) (m3)--(t5);
        \node[draw=none, fill=none] at (-4, 1) {$G_3=G -\{u_1,u_2,u_3\}:$};
        \node[ring] (k1) at (-1,0) {};
        \node[ring] (k2) at (1,0) {};
        \node[ring] (k3) at (2,0) {};
        \node[dot] (k4) at (3,0) {};
        \draw  [dotted] (m1)--(k1) (m1)--(k3) (m2)--(k3) (m3)--(k2) (m2)--(k2) (m3)--(k3);
     \draw (m3)--(k4);
        \node[dot] (d1) at (-2,-0.5) {};
        \node[dot] (d2) at (-1.5,-1) {};
        \node[dot] (d3) at (-0.5,-1) {};
        \node[dot] (d4) at (0,-0.5) {};
        \draw [dotted] (k1)--(d1) (k1)--(d2) (k1)--(d3) (k1)--(d4);
        \draw (d1)--(d2) (d3)--(d4);
        
        \node[dot] (e1) at (1.5,-1) {};
        \node[dot] (e2) at (2.5,-1) {};
        \draw  [dotted] (k3)--(e1) (k3)--(e2);
        \draw (e1)--(e2);
\end{tikzpicture}
\end{center}
\end{Example}

\begin{Corollary}\label{avariant}
Let $G$ be a Cameron-Walker graph that is not a star triangle consisting of an even number of triangles. Then
\[
\deg h_{R/I(G)}(t)=\dim R/I(G)=\alpha(G),
\]
or equivalently,
\[
a(R/I(G))=0.
\]
\end{Corollary}
\begin{Remark}
Corollary~\ref{avariant} recovers \cite[Theorem~1.1]{Hibi3}. 
Our proof via the independence polynomial is simpler than the original proof, which uses the Hilbert series.
\end{Remark}

\subsection{Degree of $h$-polynomials of edge ideals of antiregular graphs}
Antiregular graphs form an extremal class of graphs whose vertex degrees are 
as distinct as possible. Because of their recursive structure, they provide a 
natural family for studying algebraic invariants of edge ideals. 
We begin by recalling the definition of an antiregular graph.
\begin{Definition}
  Let $G=(V,E)$ be a simple graph with $|V| \geq 2$. $G$ is called an \emph{antiregular} if it has at most two vertices of the same degree. 
\end{Definition}
\begin{Theorem}
  For every integer $n \geq 2$, there exists a unique connected antiregular graph of order $n$, denoted by $A_n$, and a unique disconnected antiregular graph of order $n$, denoted by $\overline{A}_n$.
\end{Theorem}

The antiregular graphs can be defined by the following recurrence relationship \cite{Merris}
$$A_1 = K_1, A_{n} = K_1 + \overline{A}_{n-1}, n \geq 2, \text{ or}$$
$$A_1 = K_1, A_2 = K_2, A_{n} = K_1 + (K_1 \cup A_{n-2}), n \geq 2,$$
where $K_1$ is a graph with an isolated vertex only, $K_2$ is the complete graph on two vertices and $\overline{A}_n$ is the complement of $A_n$.

\begin{center}
  \begin{tikzpicture}[scale=0.9, every node/.style={draw=black, fill=white, circle, inner sep=2pt}]
    \node[draw=none, fill=none] at (-0.7, 3) {$A_1$};
    \node[label=right:$v_1$] (v1) at (0, 3) {};
    \node[draw=none, fill=none] at (1.5, 3) {$A_2$};
    \node[label=north west:$v_1$] (v1) at (2, 4) {};
    \node[label=south west:$v_2$] (v2) at (2, 2) {};
    \draw (v1) -- (v2);
 \node[draw=none, fill=none] at (3.5, 3) {$A_3$};
    \node[label=north west:$v_1$] (v1) at (4, 4) {};
    \node[label=south west:$v_2$] (v2) at (4, 2) {};
    \node[label=south east:$v_3$] (v3) at (6, 2) {};
    \draw (v1) -- (v2) -- (v3);
 \node[draw=none, fill=none] at (9.5, 3) {$A_4$};
    \node (v1) at (10, 2) {};
    \node (v2) at (8, 2) {};
    \node (v3) at (12, 2) {};
    \node (v4) at (12, 4) {};
    \draw (v1) -- (v2);
    \draw (v1) -- (v3);
    \draw (v1) -- (v4);
    \draw (v3) -- (v4);
\end{tikzpicture}\\
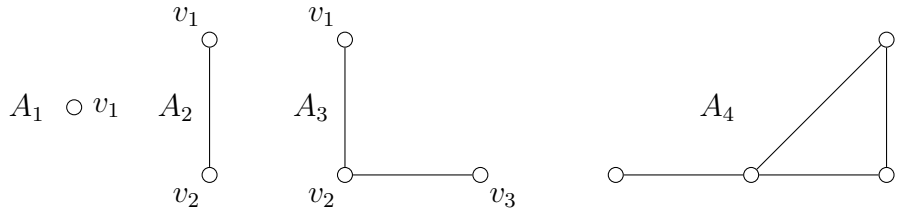
\textbf{Figure 3.} The antiregular graphs $A_1, A_2, A_3, A_4$.
\end{center}

In the following lemma, we will get the value of the independence polynomial of $A_n$ and $\overline{A}_n$ at $-1$. 
\begin{Lemma}\label{antiregular}
  Let $A_n$ be the connected antiregular graph  and $\overline{A}_n$ be the disconnected antiregular graph on $n$ vertices $(n \geq 2)$. Then,
$$I(A_n,-1) = -1, \text{ for } n \geq 2;$$
$$I(\overline{A}_n,-1) = 0 \text{ for } n \geq 2;$$
and 
$$I'(\overline{A}_n,-1) \neq 0 \text{ for } n \geq 3.$$
\end{Lemma}
\begin{proof}
  For $n = 2$, we have $I(A_2,t) = I(K_2,t) = 1+2t$, as a result $I(A_2,-1) = I(K_2,-1) =  -1$. For $n >2$, we have
$$A_{n} = K_1 + (K_1 \cup A_{n-2}).$$
By Lemma \ref{Joint}, we have
\begin{align*}\label{1}
  I(A_n,t) & = I(K_1,t) + I(K_1,t)I(A_{n-2},t) - 1 \\
   & = I(K_1,t)(1 + I(A_{n-2},t)) - 1\\
   & = (1+t)(1 + I(A_{n-2},t)) - 1.
\end{align*}
From this, it follows that 
$I(A_n, -1) = -1$, for $n> 2$.
It is clear that
$$\overline{A}_n = \overline{K_1 + \overline{A}_{n-1}} = \overline{K_1} \cup \overline{\overline{A}_{n-1}} = K_1 \cup A_{n-1}.  $$
By Lemma \ref{Joint}, we get
$$I(\overline{A}_n,t) = I(K_1,t)I(A_{n-1},t) = (1+t)I(A_{n-1},t),$$
for $n \geq 2$.
Therefore,
\[
I'(\overline{A}_n,t)
= I(A_{n-1},t) + (1+t)I'(A_{n-1},t),
\quad \text{for } n \ge 2.
\]

By substituting $t=-1$, we obtain
\[
I(\overline{A}_n,-1)=0 \quad \text{for } n \ge 2,
\]
and
\[
I'(\overline{A}_n,-1)\neq 0 \quad \text{for } n \ge 3.
\]
\end{proof}

\begin{Theorem}
Let $A_n$ denote the connected antiregular graph on $n$ vertices, and let
$\overline{A}_n$ denote the disconnected antiregular graph on $n$ vertices.
Then,
\[
\deg\ h_{R/I(A_n)}(t)=\alpha(A_n), \quad \text{for } n\ge 2,
\]
and
\[
\deg\ h_{R/I(\overline{A}_n)}(t)=\alpha(\overline{A}_n)-1,
\quad \text{for } n\ge 3.
\]
\end{Theorem}
\begin{proof}
  The conclusion follows from Lemma~\ref{antiregular} together with Theorem~\ref{chacracteristic}.
\end{proof}

\section*{Acknowledgements} \sk
The author thanks the referees for their careful reading and helpful comments.

\end{document}